\documentclass[reqno]{amsart}
\usepackage{amsthm,amsmath,amssymb}
\usepackage{stmaryrd,mathrsfs}
\usepackage{esint}

\newcommand{\ud}[0]{\,\mathrm{d}}

\newcommand{\dist}[0]{\operatorname{dist}}

\newcommand{\abs}[1]{|#1|}
\newcommand{\Babs}[1]{\Big|#1\Big|}
\newcommand{\norm}[2]{|#1|_{#2}}

\newcommand{\Bnorm}[2]{\Big|#1\Big|_{#2}}
\newcommand{\Norm}[2]{\|#1\|_{#2}}

\newcommand{\BNorm}[2]{\Big\|#1\Big\|_{#2}}
\newcommand{\pair}[2]{\langle #1,#2 \rangle}
\newcommand{\Bpair}[2]{\Big\langle #1,#2 \Big\rangle}

\newcommand{\id}[0]{\operatorname{id}}

\newcommand{\supp}[0]{\operatorname{supp}}
\newcommand{\loc}[0]{\operatorname{loc}}
\newcommand{\Id}[0]{\operatorname{id}}

\newcommand{\R}{\mathbb{R}}

\newcommand{\N}{\mathbb{N}}
\newcommand{\Z}{\mathbb{Z}}

\newcommand{\prob}[0]{\mathbb{P}}
\newcommand{\Exp}[0]{\mathbb{E}}
\newcommand{\D}[0]{\mathbb{D}}
\newcommand{\eps}[0]{\varepsilon}

\newcommand{\avL}[0]{\textit{\L}}

\newcommand{\ontop}[2]{\begin{smallmatrix} #1 \\ #2 \end{smallmatrix}}

\numberwithin{equation}{section}

\makeatletter
  \let\c@equation\c@subsection
\makeatother

\theoremstyle{plain}

\newtheorem{theorem}[subsection]{Theorem}
\newtheorem{proposition}[subsection]{Proposition}
\newtheorem{corollary}[subsection]{Corollary}
\newtheorem{lemma}[subsection]{Lemma}

\theoremstyle{definition}
\newtheorem{definition}[subsection]{Definition}

\theoremstyle{remark}
\newtheorem{remark}[subsection]{Remark}

\makeatletter
\@namedef{subjclassname@2010}{%
  \textup{2010} Mathematics Subject Classification}
\makeatother

\begin{document}

\title[Pseudo-localisation of singular integrals in $L^p$]{Pseudo-localisation of singular integrals in $L^p$}

\author[T.~P.\ Hyt\"onen]{Tuomas P.\ Hyt\"onen}
\address{Department of Mathematics and Statistics, University of Helsinki, Gustaf H\"all\-str\"omin katu 2b, FI-00014 Helsinki, Finland}
\email{tuomas.hytonen@helsinki.fi}



\maketitle

\begin{abstract}
As a step in developing a non-commutative Calder\'on--Zygmund theory, J.~Parcet (J.~Funct.\ Anal., 2009) established a new pseudo-localisation principle for classical singular integrals, showing that $Tf$ has small $L^2$ norm outside a set which only depends on $f\in L^2$ but not on the arbitrary normalised Calder\'on--Zygmund operator $T$. Parcet also asked if a similar result holds true in $L^p$ for $p\in(1,\infty)$. This is answered in the affirmative in the present paper. The proof, which is based on martingale techniques, even somewhat improves on the original $L^2$ result.
\end{abstract}

\section{Introduction}

The analogies and direct relations between the mapping properties of Calder\'on--Zygmund singular integrals and martingale transforms have well-known and far-reaching consequences. One useful property, which at first sight seems to belong to the latter class of operators only, is localisation: the supports of martingale differences are preserved by the associated martingale transforms. The regularity of the Calder\'on--Zygmund kernels, which gives an advantage in various other contexts, here seems to play against us by producing a diffusion-type effect which appears to destroy all hopes of reasonable localisation.

In view of this, the recent pseudo-localisation theorem of J.~Parcet \cite{Parcet:09} is quite remarkable. Given $f\in L^2(\R^n)$ and $s\in\N$, it provides an explicitly described set $\Sigma_{f,s}\subseteq\R^n$ (see Definition~\ref{def:Sigmafs}), so that every normalised Calder\'on--Zygmund operator $T$ maps $f$ into a function essentially concentrated on $\Sigma_{f,s}$, in the sense that \cite[Section~0.V]{Parcet:09}
\begin{equation}\label{eq:ParcetL2}
  \Big(\int_{\Sigma_{f,s}^c}\abs{Tf(x)}^2\ud x\Big)^{1/2}
  \lesssim (1+s)2^{-s\gamma/4}\Norm{f}{2},
\end{equation}
where $\gamma\in(0,1]$ is the H\"older exponent from the standard estimates. While the set $\Sigma_{f,s}$ directly obtained by the construction may easily be all of $\R^n$ for some $f\in L^2(\R^n)$, it can then be replaced by another set which still satisfies the estimate \eqref{eq:ParcetL2} and is also controlled in size, being roughly a $2^{s(1+\gamma/2n)}$-fold expansion of a cube $Q$ such that $\Norm{1_{Q^c}f}{2}\lesssim 2^{-s\gamma/4}\Norm{f}{2}$ \cite[Section~A.1]{Parcet:09}.

Perhaps surprisingly, as the boundedness of Calder\'on--Zygmund operators in $L^1(\R^n)$ usually fails, the pseudo-localisation still holds, with the same set $\Sigma_{f,s}$ and even with the faster decay
\begin{equation}\label{eq:ParcetL1}
  \int_{\Sigma_{f,s}^c}\abs{Tf(x)}\ud x
  \lesssim 2^{-s\gamma}\Norm{f}{1}.
\end{equation}
This inequality, obtained as \cite[Theorem~A.5]{Parcet:09}, is in fact far easier than \eqref{eq:ParcetL2}; the tedious almost-orthogonality estimates leading to \eqref{eq:ParcetL2} are replaced by a straightforward application of the additivity of the $L^1$ norm on disjointly supported functions. However, the procedure in $L^2(\R^n)$ of replacing $\Sigma_{f,s}$ by a set of controlled size while retaining the pseudo-localisation estimate, does not carry over to $L^1(\R^n)$ \cite[Remark~A.6]{Parcet:09}.

Motivated by the two results \eqref{eq:ParcetL2} and \eqref{eq:ParcetL1}, Parcet also asked \cite[Section~A.4]{Parcet:09} whether a pseudo-localisation principle might hold in $L^p(\R^n)$ for $p\in(1,2)$, and suggested a couple of concievable estimates in this direction. Below, I take the freedom of referring to them as ``conjectures'', although this word was not explicitly used in~\cite{Parcet:09}. Of course, given the non-linear dependence of the left sides of \eqref{eq:ParcetL2} and \eqref{eq:ParcetL1} on $f$ via the set $\Sigma_{f,s}$, no usual form of interpolation will be directly applicable. The case $p\in(2,\infty)$ was also raised in \cite[Remark A.8]{Parcet:09} as a natural question, but an intrinsic difficulty in a potential approach via duality was pointed out.

Nevertheless, in this paper, the $L^p$ estimates analogous to \eqref{eq:ParcetL2} and \eqref{eq:ParcetL1} will be established for all $p\in(1,\infty)$. The proof will deal with the full range of $p\in(1,\infty)$ at once in a unified manner, without resorting to interpolation or duality arguments. This is made possible by the use of various martingale techniques, which were originally developed to handle the difficulties arising in harmonic analysis of Banach space -valued functions in the works of Figiel \cite{Figiel:89,Figiel:90}, McConnell \cite{McConnell}, and the author~\cite{Hytonen:nonhomog}. Their successful application also to the problem at hand displays the power of these methods even in the context of classical analysis.

Besides providing this extended scope of the pseudo-localisation principle, the new proof should already be of some interest in view of the $L^2$ result \eqref{eq:ParcetL2} only. Recall that the original proof of Parcet for this estimate is, remarkably, completely ``elementary'': it only uses Cotlar's lemma and Schur's lemma, the assumptions of which are checked through a sequence of estimates involving nothing but tedious calculus for more than 25 pages \cite[pp. 528--554]{Parcet:09}. (A part of the original argument was subsequently simplified by Mei and Parcet~\cite{MeiPar}, who extended it to the case of Hilbert space -valued kernels. See the discussion after \cite[Lemma~A.2]{MeiPar}.) The present proof is somewhat shorter, admittedly at the cost of applying much deeper machinery, but I feel that the identification of these known theorems as ingredients of the pseudo-localisation principle makes the phenomena behind this result more transparent than proving it from scratch. 

The new method also yields a faster decay in the $L^2$ estimate than \eqref{eq:ParcetL2}. This is at least partially due to the directness of (here employed) Figiel's \cite{Figiel:90} approach to the $T(1)$ theorem (a version of which underlies the pseudo-localisation principle), as compared to the more usual proofs based on Cotlar's lemma: instead of attacking the operator $T$ itself, as Figiel does, the Cotlar-based approaches are in effect concerned with the estimation of $T^* T$. While in principle equivalent in $L^2$, it seems that some of the decay involved in the pseudo-localisation is lost for practical purposes in the complicated computations of the kernels for the composite operators.

\subsection*{Acknowledgements}
I would like to thank the referee for the careful reading, which led to the elimination of several typos and miscalculations. I have been supported by the Academy of Finland through projects 114374, 130166 and 133264.

\section{The set-up and the main result}

Let us agree to use the $\ell^{\infty}$ metric on $\R^n$ and denote it simply by $\abs{\cdot}$; this is more convenient than the Euclidean metric when dealing with cubes, as we will. Let $\mathscr{D}:=\bigcup_{k\in\Z}\mathscr{D}_k$ with $\mathscr{D}_k:=\{2^{-k}([0,1)^n+m):m\in\Z^n\}$ be the system of dyadic cubes in $\R^n$. A function on $\R^n$ is called $\mathscr{D}_k$-measurable if it is constant on the cubes $I\in\mathscr{D}_k$; a subset of $\R^n$ is called $\mathscr{D}_k$-measurable if it is a union of some cubes $I\in\mathscr{D}_k$. Acting on any locally integrable function $f\in L^1_{\loc}(\R^n)$, one defines the dyadic conditional expectation operators and their differences: (Note that there is a shift of the index in the present notation for the differences in comparison to Parcet's usage in ~\cite{Parcet:09}.)
\begin{equation*}
  \Exp_k f:=\sum_{I\in\mathscr{D}_k}1_I\fint_I f\ud x,\qquad
  \D_k f:=\Exp_{k+1}f-\Exp_k f,
\end{equation*}
where the integral average notation
\begin{equation*}
  \fint_I f\ud x:=\frac{1}{\abs{I}}\int_I f\ud x
\end{equation*}
was employed. In consistence with this, the notation $\avL^p(I)$ with the Polish $\avL$ will be used for the space $L^p(I)$ equipped with the normalised norm
\begin{equation*}
  \Norm{f}{\avL^p(I)}:=\Big(\fint_I\abs{f}^p\ud x\Big)^{1/p}.
\end{equation*}

We can further write
\begin{equation*}
  \D_k f=\sum_{I\in\mathscr{D}_k}1_I\D_k f=:\sum_{I\in\mathscr{D}_k}\D_I f.
\end{equation*}
The range of the projection $\D_I$ consists of functions supported on $I$, constant on $J\in\mathscr{D}$ with $\ell(J)=\frac12\ell(I)$, and with a vanishing integral. This linear space has dimension $2^n-1$. Recall the definition of the Haar functions $h^{\eta}_I$:  For $n=1$,
\begin{equation*}
  h^0_I:=\abs{I}^{-1/2} 1_I,\qquad h^1_I:=\abs{I}^{-1/2}(1_{I_{\ell}}-1_{I_r}),
\end{equation*}
where $I_{\ell}$ and $I_r$ are the left and right halves of $I$, and in general
\begin{equation*}
  h^{\eta}_I(x)=h^{(\eta_1,\ldots,\eta_n)}_{I_1\times\ldots\times I_n}(x_1,\ldots,x_n)
  :=\prod_{i=1}^n h^{\eta_i}_{I_i}(x_i),\quad\eta\in\{0,1\}^n.
\end{equation*}
It is immediate that any two of them are orthogonal. Since $h^{\eta}_I=\D_I h^{\eta}_I$ for all the $2^n-1$ choices of $\eta\in\{0,1\}^n\setminus\{0\}$, it follows that these functions form an orthonormal basis of the range of $\D_I$. Thus
\begin{equation*}
  \D_I f=\sum_{\eta\in\{0,1\}^n\setminus\{0\}}h^{\eta}_I\pair{h^{\eta}_I}{f}=:
     \sum_{\eta\in\{0,1\}^n\setminus\{0\}}\D^{\eta}_I f.
\end{equation*}
The frequently appearing summation over $\eta\in\{0,1\}^n\setminus\{0\}$, like above, will henceforth be abbreviated as $\sum_{\eta}$. The rank-one projectins $D^{\eta}_I$ satisfy $\D^{\eta}_I\D^{\theta}_J=\delta_{IJ}\delta_{\eta\theta}\D^{\eta}_I$.

For $f\in L^p(\R^n)$, $p\in(1,\infty)$, one has
\begin{equation*}
  \Exp_k f\underset{k\to\infty}{\longrightarrow} f,\qquad
  \Exp_k f\underset{k\to-\infty}{\longrightarrow} 0
\end{equation*}
both pointwise a.e.\ and in $L^p(\R^n)$, and hence
\begin{equation*}
  f=\lim_{\substack{N\to+\infty \\ M\to-\infty}}(\Exp_N f-\Exp_M f)
    =\sum_{k\in\Z}\D_k f.
\end{equation*}
Since the $\D_k f=\Exp_{k+1}f-\Exp_k f$ are martingale differences, this series converges unconditionally in $L^p(\R^n)$ by a well-known theorem of Burkholder.

\begin{remark}
The Haar expansion
\begin{equation}\label{eq:fHaar}
  f=\sum_{I\in\mathscr{D}}\sum_{\eta}\D_I^{\eta}f=\sum_{I\in\mathscr{D}}\sum_{\eta}h^{\eta}_I\pair{h^{\eta}_I}{f}
\end{equation}
is also unconditionally convergent in $L^p(\R^n)$, $p\in(1,\infty)$.
\end{remark}

\begin{proof}
For the convenience of the reader, I derive this well-known result from Burkholder's theorem. (I give an argument which is equally valid for vector-valued functions $f\in L^p(\R^n;X)$, as long as Burkholder's theorem holds in $L^p(\R^n;X)$, i.e., the Banach space $X$ is a so-called UMD space.) Let $\alpha^{\eta}_I$ be arbitrary signs, and let $\eps_I$ be independent random signs with $\prob(\eps_I=+1)=\prob(\eps_I=-1)=\frac12$; write $\Exp_{\eps}$ for the corresponding expectation. Then
\begin{equation*}
  \BNorm{\sum_{I\in\mathscr{D}}\sum_{\eta}\alpha^{\eta}_I\D^{\eta}_I f}{p}
  \leq\sum_{\eta}\BNorm{\sum_{I\in\mathscr{D}}\alpha^{\eta}_I\D^{\eta}_I f}{p}
  \lesssim\sum_{\eta}\Big(\Exp_{\eps}\BNorm{\sum_{I\in\mathscr{D}}\eps_I\D^{\eta}_I f}{p}^p\Big)^{1/p},
\end{equation*}
where the last estimate follows from Burkholder's theorem and the fact that $\D^{\eta}_I=\D_I\D^{\eta}_I$. Next, we further write
\begin{equation*}
  \D^{\eta}_I f=\D^{\eta}_I\D_I f=h_I^{\eta}\pair{h^{\eta}_I}{\D_I f}.
\end{equation*}
By using that $\abs{h^{\eta}_I}=h^0_I=1_I/\abs{I}^{1/2}$, and the fact that the distribution of $\eps_I$ and $-\eps_I$ is equal, we have
\begin{align*}
  \Exp_{\eps}\BNorm{\sum_{I\in\mathscr{D}}\eps_I h_I^{\eta}\pair{h^{\eta}_I}{\D_I f}}{p}^p
  &=\Exp_{\eps}\BNorm{\sum_{I\in\mathscr{D}}\eps_I h_I^0\pair{h^{\eta}_I}{\D_I f}}{p}^p \\
  &=\Exp_{\eps}\BNorm{\sum_{I\in\mathscr{D}}\eps_I \Exp_I\big(\abs{I}^{1/2}h^{\eta}_I\cdot\D_I f\big)}{p}^p,
\end{align*}
where $\Exp_I \phi:=1_I\fint_I\phi\ud x$. Thanks to Stein's inequality for the expectation operators~$\Exp_I$, $I\in\mathscr{D}$, followed by Kahane's contraction principle applied to the functions $\abs{I}^{1/2}h^{\eta}_I$ which are bounded in absolute value by $1$, we may continue the estimate with
\begin{equation*}
  \lesssim\Exp_{\eps}\BNorm{\sum_{I\in\mathscr{D}}\eps_I \abs{I}^{1/2}h^{\eta}_I\cdot\D_I f}{p}^p
  \leq\Exp_{\eps}\BNorm{\sum_{I\in\mathscr{D}}\eps_I \D_I f}{p}^p.
\end{equation*}
Another application of Burkholder's theorem shows that this is bounded by $\Norm{f}{p}^p$, completing the proof.
\end{proof}

Thus the collection $\{h^{\eta}_I:I\in\mathscr{D},\eta\in\{0,1\}^n\setminus\{0\}\}$, which is an orthonormal basis of $L^2(\R^n)$, is also an unconditional basis of $L^p(\R^n)$, $p\in(1,\infty)$, so the series in \eqref{eq:fHaar} sums up to the same limit irrespective of the summation order. Note that the non-cancellative Haar functions $h^0_I$ do not appear as part of this basis, but they are still handy for related considerations.

Given $I\in\mathscr{D}$ and $s\in\N$, the notation $I^{(s)}$ will stand for the $s$th dyadic ancestor of $I$, i.e., the unique $I^{(s)}\in\mathscr{D}$ such that $I^{(s)}\supseteq I$ and $\ell(I^{(s)})=2^s\ell(I)$, where $\ell(I)$ is the side-lenght of $I$. For $I\in\mathscr{D}$ and $m\in\Z^n$, the notation $I\dot+m$ indicates the dyadic cube of the same size obtained by translating $I$ by $m$ times its side-length, i.e., $I\dot+m:=I+\ell(I)m$.

\begin{definition}[$\Sigma_{f,s}$ and related sets, {\cite[p.~517]{Parcet:09}}]\label{def:Sigmafs} 
Let $f\in L^1_{\loc}(\R^n)$ and $s\in\N:=\{0,1,2,\ldots\}$. For each $k\in\Z$, let $\Omega_k$ (its dependence on $f$ and $s$ is supressed from the notation) be the smallest $\mathscr{D}_k$-measurable set which contains the support of $\D_{k+s}f$, i.e., $\Omega_k$ is the union of those cubes $I\in\mathscr{D}_k$ where $\D_{k+s}f$ is not identically zero. Let $9\Omega_k$ be the union of the corresponding concentric $9$-fold expansions $9I$, which is still $\mathscr{D}_k$-measurable. (The factor $9$ is important for this last conclusion; $8$ or $10$ would not do.) Then finally
\begin{equation*}
  \Sigma_{f,s}:=\bigcup_{k\in\Z}9\Omega_k.
\end{equation*}
\end{definition}

A standard kernel $K(x,y)$ is a function on $\R^n\times\R^n\setminus\{x=y\}$ with the estimates
\begin{equation*}
  \abs{K(x,y)}\leq\frac{C}{\abs{x-y}^n},\quad x\neq y,\qquad\text{and}
\end{equation*}
\begin{equation*}
\begin{split}
  \abs{K(x+h,y)-K(x,y)}+&\abs{K(x,y+h)-K(x,y)} \\
    &\leq\frac{C\abs{h}^{\gamma}}{\abs{x-y}^{n+\gamma}},\quad
  \abs{x-y}>2\abs{h},
\end{split}
\end{equation*}
where $\gamma\in(0,1]$ is a fixed parameter. The kernel is said to be normalised if these estimates hold with $C=1$.
For this paper, a Calder\'on--Zygmund operator is an operator $T$ sending $f\in L^p(\R^n)$, $p\in[1,\infty)$, to the function $Tf$ defined on $(\supp f)^c$ by the formula
\begin{equation*}
  Tf(x)=\int_{\R^n}K(x,y)f(y)\ud y,\quad x\notin\supp f.
\end{equation*}
Recall that the support of a measurable function $f$ is the complement of the union of all balls in which $f$ vanishes almost everywhere; hence it is a closed set.

As it turns out, this weak definition suffices for the pseudo-locali\-sation principle. Note that $\supp f\subseteq\Sigma_{f,s}$, so that the formula \eqref{eq:ParcetL2} only involves $Tf(x)$ for $x\notin\supp f$, and this assertion makes perfect sense without even having $Tf$ globally defined. Besides being defined pointwise, Calder\'on--Zygmund operators automatically enjoy the following off-diagonal boundedness property:

\begin{remark}\label{rem:CZoffdiag}
Let $F$ be a closed set of $\R^n$ and $K$ a compact set disjoint from $F$ (and hence at a positive distance from $F$). Given any Calder\'on--Zygmund operator $T$, the operator $1_K T$ maps $L^p(F):=\{\phi\in L^p(\R^n):\supp\phi\subseteq F\}$ boundedly into $L^p(\R^n)$ for $p\in[1,\infty)$.
\end{remark}

\begin{proof}
Let $\delta:=\dist(K,F)>0$. Then for $x\in K$ and $f\in L^p(F)$,
\begin{align*}
  \abs{Tf(x)}
  &\leq\Big(\int_F\abs{K(x,y)}^{p'}\ud y\Big)^{1/p'}\Norm{f}{p} \\
  &\lesssim\Big(\int_{\abs{y-x}\geq\delta}\frac{C\ud y}{\abs{x-y}^{np'}}\Big)^{1/p'}\Norm{f}{p}
   \lesssim\delta^{-n/p}\Norm{f}{p},
\end{align*}
and hence $\Norm{1_K Tf}{p}\lesssim\delta^{-n/p}\abs{K}^{1/p}\Norm{f}{p}$.
\end{proof}

If $T$ is a Calder\'on--Zygmund operator for which $Tf(x)$ is also defined for a.e.\ $x\in\supp f$, in such a way that the mapping $T:f\mapsto Tf$ is linear and bounded on $L^p(\R^n)$ for one (and then all) $p\in(1,\infty)$, then $T$ is said to be a bounded Calder\'on--Zygmund operator. It is called normalised if $\Norm{Tf}{2}\leq\Norm{f}{2}$.

\begin{theorem}\label{thm:pseudoloc}
Let $p\in(1,\infty)$, $f\in L^p(\R^n)$, $s\in\N$. Then every Calder\'on--Zygmund operator $T$ with a normalised kernel satisfies
\begin{equation}\label{eq:pseudoloc}
  \Big(\int_{\Sigma_{f,s}^c}\abs{Tf(x)}^p\ud x\Big)^{1/p}
  \lesssim (1+s) 2^{-s\min(\gamma,1/2,1/p')}\Norm{f}{p}.
\end{equation}
If, moreover, $T$ is bounded and normalised, this estimate also holds with $\Sigma_{f,s}$ replaced by
\begin{equation}\label{eq:Qfs}
  100\cdot 2^{s[1+\min(\gamma,1/2,1/p')\cdot p'/n]}Q_{f,s},
\end{equation}
where $Q_{f,s}$ is any cube such that
\begin{equation*}
  \Norm{1_{Q_{f,s}^c}f}{p}\leq (1+s)2^{-s\min(\gamma,1/2,1/p')}\Norm{f}{p}.
\end{equation*}
\end{theorem}

Here and below, the notation $A\lesssim B$ stands for $A\leq CB$, where the constant $C$ is only allowed to depend on the dimension $n$, the Lebesgue exponent $p$, and the H\"older exponent $\gamma$, but never on $f$, $s$, or $T$.

Note that \eqref{eq:pseudoloc} is stronger than Parcet's conjecture \cite[(A.3)]{Parcet:09}, where the decay exponent involved the product, rather than the minimum, of the three small numbers $\gamma,1/2,1/p'\in(0,1]$. Theorem~\ref{thm:pseudoloc} fails to prove, however, the unnumbered displayed formula preceding \cite[(A.3)]{Parcet:09}, which was suggested by na\"ive interpolation between Parcet's estimates \eqref{eq:ParcetL2} and \eqref{eq:ParcetL1}. Indeed the decay exponent given by \eqref{eq:pseudoloc} vanishes in the limit $p\to 1$, rather than approaching the $L^1$ decay rate of~\eqref{eq:ParcetL1}.

The heart of the matter is the bound \eqref{eq:pseudoloc} concerning the set $\Sigma_{f,s}$. Once this estimate is obtained, the variant with $Q_{f,s}$ follows straightforwardly by essentially repeating the argument of \cite[Section~A.1]{Parcet:09}.

We then turn to the proof of the main estimate \eqref{eq:pseudoloc}.

\section{Reduction to an operator boundedness problem}

In this section, the pseudo-localisation estimate involving the restricted operator $1_{\Sigma_{f,s}^c}T$ will be reduced to a new question concerning the $L^p(\R^n)$ boundedness of certain globally defined operators derived from $T$ and $\Sigma_{f,s}^c$. This still essentially follows the argument of Parcet from the $L^2(\R^n)$ case \cite[Sections 2.2--2.3]{Parcet:09}. To begin with, the following technical lemma will save some trouble of worrying about the convergence issues in the coming manipulations. It has a reasonably standard flavour, but recall that the $L^p(\R^n)$ boundedness of $T$ is not assumed, which makes the reasoning slightly more complicated.

\begin{lemma}\label{lem:redFiniteHaar}
It suffices to prove the pseudo-localisation estimate for all $f$ with a finite Haar expansion.
\end{lemma}

\begin{proof}
Let $f\in L^p(\R^n)$ and consider the closed set $F:=\bigcup_{k\leq 0}\bar{\Omega}_k\cup\supp f\subseteq\Sigma_{f,s}$. 
Fix a compact $K$ disjoint from $F$, and let $F':=\overline{F+\frac12\dist(F,K)[-1,1]^n}$, which is still separated from $K$. Note that $\supp\Exp_k f\subseteq\supp f+2^{-k}[-1,1]^n\subseteq F'$ for all large $k$, while
\begin{equation*}
  \supp\Exp_k f=\supp\sum_{j<k}\D_j f
  \subseteq\bigcup_{j<k}\supp\D_{(j-s)+s}f
  \subseteq\bigcup_{j<k}\bar{\Omega}_{j-s}f
  \subseteq F\subseteq F'
\end{equation*}
for all $k\leq s+1$. Thus the functions $1_E(\Exp_a f-\Exp_b f)$, with $a\geq a_0>s\geq b$ and $E$ a bounded $\mathscr{D}_b$-measurable set, have a finite Haar expansion, belong to $L^p(F')$, and converge to $f$ in this space as $a\to\infty$, $b\to-\infty$, and $E\uparrow\R^n$.

Denoting by $\tilde{f}$ one of these approximations, and assuming the pseudo-localisation for functions with a finite Haar expansion, there holds
\begin{equation*}
  \Norm{1_{\Sigma_{f,s}^c}1_K T\tilde{f}}{p}
  \leq\Norm{1_{\Sigma_{\tilde{f},s}^c} T\tilde{f}}{p}
  \lesssim(1+s) 2^{-s\min(\gamma,1/2,1/p')}\Norm{\tilde{f}}{p},
\end{equation*}
where the first estimate follows from the fact that $\Sigma_{\tilde{f},s}\subseteq\Sigma_{f,s}$, since $\supp\D_{k+s}\tilde{f}\subseteq\supp\D_{k+s}f$. Letting $\tilde{f}\to f$ along the family of functions as considered, and using the continuity of $1_K T:L^p(F')\to L^p(\R^n)$ (Remark~\ref{rem:CZoffdiag}), it follows that
\begin{equation*}
  \Norm{1_{\Sigma_{f,s}^c}1_K Tf}{p}
  \lesssim(1+s) 2^{-s\min(\gamma,1/2,1/p')}\Norm{f}{p}.
\end{equation*}
As $K\uparrow F^c\supseteq\Sigma_{f,s}^c$, this gives the pseudo-localisation estimate for $f$.
\end{proof}

Let $f$ be henceforth a function with a finite Haar expansion. The object to be estimated can then be written as
\begin{equation*}
  1_{\Sigma_{f,s}^c}Tf
  =1_{\Sigma_{f,s}^c}\Big(\sum_k\Exp_k T\D_{k+s}f
    +\sum_k(\Id-\Exp_k)T\D_{k+s}f\Big).
\end{equation*}

Using the facts that $9\Omega_k\subseteq\Sigma_{f,s}$ is $\mathscr{D}_k$-measurable (so that the multiplication operator of its indicator commutes with $\Exp_k$), and that the  distance of the sets $(9\Omega_k)^c$ and $\Omega_k\supseteq\supp\D_{k+s}f$ is $4\cdot 2^{-k}$, we have
\begin{align*}
  1_{\Sigma_{f,s}^c} &\sum_k(\Id-\Exp_k)T\D_{k+s}f
  =1_{\Sigma_{f,s}^c}\sum_k (\Id-\Exp_k)1_{(9\Omega_k)^c}T\D_{k+s}f \\
  &=1_{\Sigma_{f,s}^c}\sum_k (\Id-\Exp_k)1_{(9\Omega_k)^c}T_{4\cdot 2^{-k}}\D_{k+s}f \\
  &=1_{\Sigma_{f,s}^c}\sum_k (\Id-\Exp_k)T_{4\cdot 2^{-k}}\D_{k+s}f,
\end{align*}
where $T_{\eps}$ is the truncated singular integral
\begin{equation*}
  T_{\eps}g(x)=\int_{\abs{y-x}>\eps}K(x,y)g(y)\ud y,
\end{equation*}
which is automatically globally defined on $L^p(\R^n)$.
Putting the previous equalities together gives Parcet's decomposition
\begin{equation}\label{eq:defPsi}
\begin{split}
  1_{\Sigma_{f,s}^c}Tf
  &=1_{\Sigma_{f,s}^c}\Big(\sum_k\Exp_k T\D_{k+s}f
    +\sum_k(\Id-\Exp_k)T_{4\cdot 2^{-k}}\D_{k+s}f\Big) \\
  &=:1_{\Sigma_{f,s}^c}\big(\Phi_s f+\Psi_s f\big).
\end{split}
\end{equation}

Let us have a closer look at the first term by expanding the operators $\Exp_k$ and $\D_{k+s}$ in terms of the Haar functions:
\begin{equation*}
\begin{split}
  \Phi_s f
  &=\sum_{\ontop{I,J\in\mathscr{D}}{\ell(J)=2^s\ell(I)}}\sum_{\eta}
    h^0_J\pair{h^0_J}{Th^{\eta}_I}\pair{h^{\eta}_I}{f} \\
  &=\sum_{m\in\Z^n}\sum_{\eta}\sum_{I\in\mathscr{D}}
    h^0_{I^{(s)}\dot+m}\pair{h^0_{I^{(s)}\dot+m}}{Th^{\eta}_I}\pair{h^{\eta}_I}{f}.
\end{split}
\end{equation*}
Observe that, if $\pair{h^{\eta}_I}{f}\neq 0$, then $I^{(s)}\subseteq\Sigma_{f,s}$, and hence $1_{\Sigma_{f,s}^c}h^0_{I^{(s)}}=0$. Thus one can virtually subtract some terms without affecting the value of $1_{\Sigma_{f,s}^c}\Phi_s f$, to the result that
\begin{equation}\label{eq:defPhi}
\begin{split}
  &1_{\Sigma_{f,s}^c}\Phi_s f \\
  &=1_{\Sigma_{f,s}^c}\sum_{m\in\Z^n}\sum_{\eta}\sum_{I\in\mathscr{D}}
    \big(h^0_{I^{(s)}\dot+m}-h^0_{I^{(s)}}\big)\pair{h^0_{I^{(s)}\dot+m}}{Th^{\eta}_I}\pair{h^{\eta}_I}{f} \\
  &=: 1_{\Sigma_{f,s}^c}\tilde\Phi_s f,
\end{split}
\end{equation}
and therefore
\begin{equation}\label{eq:sumPhiPsi}
  1_{\Sigma_{f,s}^c}T f=1_{\Sigma_{f,s}^c}\big(\tilde\Phi_s f+\Psi_s f\big).
\end{equation}
To prove the pseudo-localisation estimate, it hence suffices to bound the operator norms of $\tilde\Phi_s$ and $\Psi_s$ appropriately, which will be the concern of the following two sections.

Let us notice, although it will not be used here, that the replacement of $\Phi_s$ by $\tilde\Phi_s$ was in effect the removal of a paraproduct associated with $T^* 1$, and these operators agree globally  (rather than just on $\Sigma_{f,s}^c$) in case $T^* 1=0$. See \cite[Sec.~2.3]{Parcet:09}.

\section{The operator $\tilde\Phi_s$}

This section is devoted to the analysis of the operator $\tilde\Phi_s$. When $s=0$, the following bound already appears as part of Figiel's~\cite{Figiel:90} proof of the $T(1)$ theorem. I will be able to exploit some intermediate results of his proof, but obtaining the decay in $s$ will also depend on some new estimates.

\begin{proposition}\label{prop:shiftedT1}
Let $p\in(1,\infty)$ and $T$ be a Calder\'on--Zygmund operator with a normalised kernel. Then the operator $\tilde\Phi_s$ defined in \eqref{eq:defPhi} satisfies
\begin{equation*}
  \Norm{\tilde\Phi_s}{p\to p}
  \lesssim(1+s) 2^{-s\min(\gamma,1/2,1/p')}.
\end{equation*}
\end{proposition}

It is convenient to start by controlling the Haar coefficients of $T$ appearing in the definition~\eqref{eq:defPhi} of $\tilde\Phi_s$. Similar estimates of course appear in Parcet's paper \cite{Parcet:09}, but it seems that expressing the bounds in terms of the dyadic cubes rather than some reference points inside them will simplify the presentation. Note that one may take the outer summation in~\eqref{eq:defPhi} over $m\in\Z^n\setminus\{0\}$ only, since the first factor of the summand vanishes for $m=0$.

\begin{remark}\label{rem:3IminusI}
For any cube $I$, we have
$\displaystyle  \int_{3I\setminus I}\int_I\frac{\ud y\ud x}{\abs{x-y}^n}\lesssim\abs{I}.$
\end{remark}

\begin{proof}
By a change-of-variable and an obvious decomposition of $3I\setminus I$, it suffices to prove this for $I=[0,1)^n$ and an adjacent unit cube in place of $3I\setminus I$. Let $i$ be one of the coordinate directions in which this adjacent cube projects onto a unit interval $[a,a+1)$ with $a\neq 0$ (thus $a\in\{-1,+1\}$, and without loss of generality $a=1$), and consider first the integral with respect to $\ud y_i\ud x_i$. We write $x',y'$ for the $(n-1)$-vectors obtained from $x,y$ by deleting the entries $x_i,y_i$, and observe that $\abs{x-y}\eqsim\abs{x'-y'}+\abs{x-y}$. Then, for $n\geq 3$,
\begin{align*}
  \int_0^1\Big( \int_1^2 &\frac{\ud x_i}{[\abs{x'-y'}+(x_i-y_i)]^n}\Big) \ud y_i \\
  &\leq\int_0^1\frac{(n-1)^{-1}}{[\abs{x'-y'}+(1-y_i)]^{n-1}}\ud y_i \\
  &\leq (n-1)^{-1}(n-2)^{-1}\abs{x'-y'}^{-(n-2)}.
\end{align*}
This may be integrated with respect to $x'\in[0,1)^{n-1}$ with a finite bound which may be taken independent of $y'$, since the remaining singularity has a lower order $n-2$ than the remaining dimension $n-1$. For $n\in\{1,2\}$, the argument needs an elementary modification, since the logarithm appears as a primitive from the integrals above; this is left to the reader.
\end{proof}

\begin{lemma}\label{lem:HaarEst}
For $I,J\in\mathscr{D}$ with $\ell(I)\leq\ell(J)$ and $I\cap J=\varnothing$,
\begin{equation*}
\begin{split}
  &\abs{\pair{h^0_J}{Th^{\eta}_I}} \\
  &\lesssim\Big(\frac{\ell(I)}{\ell(J)}\Big)^{n/2}
  \begin{cases}
    \ell(I)^{\gamma}\ell(J)^n\dist(I,J)^{-n-\gamma}, & \dist(I,J)\geq\ell(J), \\
    \ell(I)^{\gamma}\dist(I,\partial J)^{-\gamma}, & \ell(I)\leq\dist(I,J)\leq\ell(J), \\
    1, & \dist(I,J)\leq\ell(I).
  \end{cases}
\end{split}
\end{equation*}
\end{lemma}

\begin{proof}
If $\dist(I,J)\geq\ell(J)$, then, writing $y_I$ for the centre of $I$,
\begin{equation*}
\begin{split}
  \abs{\pair{h^0_J}{Th^{\eta}_I}}
  &=\Babs{\iint h^0_J(x)[K(x,y)-K(x,y_I)]h^{\eta}_I(y)\ud y\ud x} \\
  &\lesssim\frac{\ell(I)^{\gamma}}{\dist(J,I)^{n+\gamma}}\Norm{h^0_J}{1}\Norm{h^{\eta}_I}{1}
   \eqsim\frac{\ell(I)^{\gamma+n/2}\ell(J)^{n/2}}{\dist(J,I)^{n+\gamma}}.
\end{split}
\end{equation*}
Estimating slightly differently when $\ell(I)\leq\dist(I,J)\leq\ell(J)$,
\begin{equation*}
\begin{split}
  \abs{\pair{h^0_J}{Th^{\eta}_I}}
  &\lesssim\Norm{h^0_J}{\infty}\Norm{h^{\eta}_I}{1}
    \int_J\frac{\ell(I)^{\gamma}}{\dist(x,I)^{n+\gamma}}\ud x \\
  &\lesssim\Big(\frac{\ell(I)}{\ell(J)}\Big)^{n/2}\int_{\dist(J,I)}^{\infty}
    \frac{\ell(I)^{\gamma}}{t^{n+\gamma}}t^{n-1}\ud t
   \lesssim\frac{\ell(I)^{n/2+\gamma}}{\ell(J)^{n/2}\dist(J,I)^{\gamma}}.
\end{split}
\end{equation*}

Finally, for any disjoint position of $I$ and $J$ (but this estimate will be used only when $\dist(I,J)\leq\ell(I)$), we have the following bounds involving Remark~\ref{rem:3IminusI}:
\begin{equation*}
\begin{split}
  &\abs{\pair{h^0_J}{Th^{\eta}_I}} \\
  &\leq\abs{\pair{h^0_J 1_{3I}}{Th^{\eta}_I}}
    +\Babs{\iint h^0_J1_{(3I)^c}(x)[K(x,y)-K(x,y_I)]h^{\eta}_I(y)\ud y\ud x} \\
  &\lesssim\Norm{h^0_J}{\infty}\Norm{h^{\eta}_I}{\infty}\Big(\int_{3I\setminus I}\int_I\frac{\ud y\ud x}{\abs{x-y}^n}+
    \int_{(3I)^c}\frac{\ell(I)^{\gamma}}{\dist(x,I)^{n+\gamma}}\abs{I}\ud x\Big) \\
  &\lesssim\abs{J}^{-1/2}\abs{I}^{-1/2}\Big(\abs{I}
     +\abs{I}\int_{\ell(I)}^{\infty}\frac{\ell(I)^{\gamma}}{t^{n+\gamma}}t^{n-1}\ud t\Big)
  \lesssim\Big(\frac{\ell(I)}{\ell(J)}\Big)^{n/2}.
\end{split}
\end{equation*}
Combining these estimates gives the assertion.
\end{proof}

The above bound is not very good when the smaller cube is close to the boundary of the bigger one. This is a common source of pain in related considerations, and different methods have been devised to overcome it in various situations. In the present case it will suffice to obtain the following average bound, which exhibits required decay. For $J\in\mathscr{D}$ and $m\in\Z^n$, let $J\dot+m:=J+\ell(J)m$ be the dyadic cube translated in each direction by a multiple of its side-length.

\begin{lemma}\label{lem:averEst}
For $I,J\in\mathscr{D}$ with $\ell(I)\leq\ell(J)$ and $m\in\Z^n\setminus\{0\}$, there holds
\begin{equation*}
\begin{split}
  &\BNorm{\sum_{\ontop{K\subseteq J}{\ell(K)=\ell(I)}}\pair{h^0_{J\dot+m}}{Th^{\eta}_K}h^{\eta}_K}{\avL^r(J)} \\
  &\lesssim\abs{J}^{-1/2}\Big(\frac{\ell(I)}{\ell(J)}\Big)^{\min(\gamma,1/r)}
    \Big(1+\log\frac{\ell(J)}{\ell(I)}\Big)^{\delta_{\gamma,1/r}\cdot 1/r}(1+\abs{m})^{-n-\gamma},
\end{split}
\end{equation*}
where $\delta_{\gamma,1/r}$ is Kronecker's delta, i.e., $1$ if $\gamma=1/r$ and $0$ otherwise.
\end{lemma}

\begin{proof}
If $m\notin\{-1,0,1\}^n$, then $\ell(J)\leq\dist(K,J\dot+m)\eqsim\ell(J)\abs{m}$ for $K\subseteq J$, so all the pairings with $T$ are bounded by
\begin{equation*}
  \ell(I)^{n/2+\gamma}\ell(J)^{-n/2-\gamma}\abs{m}^{-n-\gamma}.
\end{equation*}
Also the $h^{\eta}_K$ are disjointly supported and bounded by $\abs{K}^{-1/2}=\ell(I)^{-n/2}$. Hence even the $L^{\infty}(J\dot+m)$ norm, and thus the $\avL^r(J\dot+m)$ norm, is dominated 
by the product of these numbers, which is exactly as claimed.

If $m\in\{-1,0,1\}^n\setminus\{0\}$, so that $J\dot+m$ and $J$ are adjacent, then one observes that there are $O((\ell(J)/\ell(I))^{n-1})$ cubes $K$ with
\begin{equation*}
  \dist(K,\partial(J\dot+m))=k\ell(I)
\end{equation*}
for each $k=0,\ldots,\ell(J)/\ell(I)$, and hence
\begin{equation*}
\begin{split}
  &\BNorm{\sum_{\ontop{K\subseteq J}{\ell(K)=\ell(I)}}\pair{h^0_{J\dot+m}}{Th^{\eta}_K}h^{\eta}_K}{\avL^r(J)} \\
  &\lesssim\Big(\frac{\ell(I)}{\ell(J)}\Big)^{n/2}\Big\{\frac{1}{\abs{J}}
    \Big(\frac{\ell(J)}{\ell(I)}\Big)^{n-1}\sum_{k=0}^{\ell(J)/\ell(I)}
    \Big(\frac{\ell(I)}{(1+k)\ell(I)}\Big)^{\gamma r}\frac{\abs{I}}{\abs{I}^{r/2}}\Big\}^{1/r} \\
  &=\ell(J)^{-n/2}\Big\{
    \frac{\ell(I)}{\ell(J)}\sum_{k=0}^{\ell(J)/\ell(I)}(1+k)^{-\gamma r}\Big\}^{1/r} \\
  &\lesssim\ell(J)^{-n/2}\Big\{
    \frac{\ell(I)}{\ell(J)}\Big(\frac{\ell(J)}{\ell(I)}\Big)^{(1-\gamma r)_+}
     \Big(1+\log\frac{\ell(J)}{\ell(I)}\Big)^{\delta_{\gamma r,1}}\Big\}^{1/r} \\
  &\lesssim\ell(J)^{-n/2}\Big(\frac{\ell(I)}{\ell(J)}\Big)^{\min(\gamma,1/r)}
      \Big(1+\log\frac{\ell(J)}{\ell(I)}\Big)^{\delta_{\gamma,1/r}\cdot 1/r},
\end{split}
\end{equation*}
which is again as claimed.
\end{proof}

Now we can start with the proof of Proposition~\ref{prop:shiftedT1} where, we recall,
\begin{equation}\label{eq:ETD2}
  \tilde\Phi_s f
  =\sum_{m\in\Z^n\setminus\{0\}}\sum_{\eta}\sum_{I\in\mathscr{D}}
    \big(h^0_{I^{(s)}\dot+m}-h^0_{I^{(s)}}\big)\pair{h^0_{I^{(s)}\dot+m}}{Th^{\eta}_I}\pair{h^{\eta}_I}{f}.
\end{equation}

Let us denote by $\Lambda_{s,m}$ and $U_m$ the linear operators acting on the Haar basis as follows:
\begin{equation*}
  \Lambda_{s,m}:h^{\eta}_I\mapsto\pair{h^0_{I^{(s)}\dot+m}}{Th^{\eta}_I}h^{\eta}_{I^{(s)}},\quad
  U_m:h^{\eta}_I\mapsto h^0_{I\dot+m}-h^0_I.
\end{equation*}
Then \eqref{eq:ETD2} says that
\begin{equation*}
  \tilde\Phi_s f=\sum_{m\in\Z^n}U_m\Lambda_{s,m}f.
\end{equation*}
The operators $U_m$ were considered by Figiel, who showed  \cite[Theorem~1]{Figiel:89} that
\begin{equation*}
  \Norm{U_m}{p\to p}\lesssim\log(2+\abs{m}).
\end{equation*}
The proof of Proposition~\ref{prop:shiftedT1} is obviously completed once it is shown that the operators $\Lambda_{s,m}$ satisfy
\begin{equation}\label{eq:shiftedT1mainEst}
  \Norm{\Lambda_{s,m}}{p\to p}\lesssim(1+s) 2^{-s\min(\gamma,1/2,1/p')}(1+\abs{m})^{-n-\gamma},
\end{equation}
since
\begin{equation*}
  \sum_{m\in\Z^n}\frac{\log(2+\abs{m})}{(1+\abs{m})^{n+\gamma}}\lesssim 1.
\end{equation*}
So let us turn to this task.

Let $s$, $m$ and $\eta$ be fixed and abbreviate $\lambda_I:=\pair{h^0_{I^{(s)}\dot+m}}{Th^{\eta}_I}$, $\alpha_I:=\pair{h^{\eta}_I}{f}$ and $h_I:=h_I^{\eta}$. Then $\Lambda_{s,m}f$ is a sum of $2^n-1$ series (corresponding to the different values of $\eta$) of the form
\begin{equation*}
  \sum_{J\in\mathscr{D}}\sum_{\ontop{I\subseteq J}{\ell(I)=2^{-s}\ell(J)}}\lambda_I\alpha_I h_J.
\end{equation*}
By the unconditionality of the Haar basis, the $L^p$ norm of this quantity is comparable to the following, where the $\eps_J$ designate independent random signs on some probability space $(\Omega,\prob)$, with the distribution $\prob(\eps_J=-1)=\prob(\eps_J=1)=\frac12$, and $\Exp_{\eps}$ is the related expectation operator:
\begin{equation}\label{eq:shiftedT1toEst}
\begin{split}
  \Big(\Exp_{\eps} &\BNorm{\sum_{J\in\mathscr{D}}\eps_J\sum_{\ontop{I\subseteq J}{\ell(I)=2^{-s}\ell(J)}}
   \lambda_I\alpha_I h_J}{p}^p\Big)^{1/p} \\
  &=\Big(\int_{\R^n}\Exp_{\eps}\Babs{\sum_{J\in\mathscr{D}} \eps_J
   \sum_{\ontop{I\subseteq J}{\ell(I)=2^{-s}\ell(J)}}\lambda_I\alpha_I \frac{1_J(x)}{\abs{J}^{1/2}}}^p
     \ud x\Big)^{1/p},
\end{split}
\end{equation}
where we used the pointwise equality $\abs{h_J(x)}=h^0_J(x)=1_J(x)/\abs{J}^{1/2}$ and the fact that the possible change of sign does not affect the randomised norms with the multiplicative random sign $\eps_J$ in front in any case. (Note, however, that $h_J=h_J^{\eta}$ for some $\eta\neq 0$.)

Consider the above integrand for a fixed $x\in\R^n$, introducing auxiliary variables $y_J\in J$ for each $J\in\mathscr{D}$. By the orthogonality relations of the Haar functions,
\begin{equation}\label{eq:shiftedT1pointwise}
\begin{split}
  &\Exp_{\eps}\Babs{\sum_{J\in\mathscr{D}} \eps_J
   \sum_{\ontop{I\subseteq J}{\ell(I)=2^{-s}\ell(J)}}\lambda_I\alpha_I \frac{1_J(x)}{\abs{J}^{1/2}}}^p \\
  &=\Exp_{\eps}\Babs{\sum_{J\in\mathscr{D}} \eps_J
   \fint_J\Big(\abs{J}^{1/2}\sum_I\lambda_I h_I(y_J)\Big)\Big(\sum_I\alpha_I h_I(y_J)\Big)\ud y_J 1_J(x)}^p,
\end{split}
\end{equation}
where the summation condition on $I$ is as before.

We make use of the following estimate, where $(S,\nu)$ is an abstract $\sigma$-finite measure space:

\begin{lemma}\label{lem:pullOutNorms}
Let $p\in[1,\infty)$ and $t\in[1,2]$. For $f_j\in L^t(\nu)$, $\phi_j\in L^{t'}(\nu)$, there holds
\begin{equation*}
  \Big(\Exp_{\eps}\Babs{\sum_j\eps_j\int_S\phi_j(s)f_j(s)\ud\nu(s)}^p\Big)^{1/p}
  \lesssim\sup_j\Norm{\phi_j}{t'}\Exp_{\eps}\BNorm{\sum_j\eps_j f_j}{t}.
\end{equation*}
\end{lemma}

\begin{proof}
Using the equivalence of the randomised and quadratic sums (i.e., the well-known Khintchine inequality), the left side is comparable to
\begin{equation*}
\begin{split}
  \Big(\sum_j &\Babs{\int_S\phi_j f_j\ud\nu}^2\Big)^{1/2} \\
  &\leq\sup\Big\{\sum_j\int_S\abs{\phi_j}\abs{f_j}\ud\nu\abs{\lambda_j}:\Big(\sum_j\abs{\lambda_j}^2\Big)^{1/2}\leq 1\Big\}.
\end{split}
\end{equation*}
Then
\begin{equation*}
  \sum_j\int_S\abs{\phi_j}\abs{f_j}\ud\nu\abs{\lambda_j}
  \leq\BNorm{\Big(\sum_j\abs{\lambda_j\phi_j}^2\Big)^{1/2}}{t'}
    \BNorm{\Big(\sum_j\abs{f_j}^2\Big)^{1/2}}{t}
\end{equation*}
where the second factor is comparable to the second factor in the assertion. Since $t'\geq 2$, using the triangle inequality in $L^{t'/2}$ for the first factor, it is estimated by
\begin{equation*}
  \Big(\sum_j\Norm{\abs{\lambda_j}^2\abs{\phi_j}^2}{t'/2}\Big)^{1/2}
  \leq\sup_j\Norm{\phi_j}{t'}\Big(\sum_j\abs{\lambda_j}^2\Big)^{1/2}\leq\sup_j\Norm{\phi_j}{t'}.
\end{equation*}
This completes the proof.
\end{proof}

Lemma~\ref{lem:pullOutNorms} is applied to \eqref{eq:shiftedT1pointwise} (for each fixed $x$) using the product measure space $S:=\prod_{J\in\mathscr{D}}J$, which is equipped with the product of the normalised Lebesgue measures restricted to each $J$ (i.e., exactly the measures with respect to which one integrates in \eqref{eq:shiftedT1pointwise}). This gives
\begin{equation}\label{eq:pullOutNorms}
\begin{split}
  LHS\eqref{eq:shiftedT1pointwise}
  \lesssim\sup_{J\in\mathscr{D}}
    &\BNorm{\abs{J}^{1/2}\sum_{\ontop{I\subseteq J}{\ell(I)=2^{-s}\ell(J)}}\lambda_I h_I}{\avL^{t'}(J)} \\
  &\times\Exp_{\eps}\BNorm{\sum_{J\in\mathscr{D}}\eps_J
    \sum_{\ontop{I\subseteq J}{\ell(I)=2^{-s}\ell(J)}}\alpha_I h_I(y_J)1_I(x)}{L^t(S)},
\end{split}
\end{equation}
where we choose
\begin{equation*}
  t:=\min(2,p).
\end{equation*}
By Lemma~\ref{lem:averEst} (recalling the definition of the coefficients $\lambda_I$), the first factor above is bounded by
\begin{equation*}
\begin{split}
  (1+s)^{\delta_{\gamma,1/t'}\cdot 1/t'} &2^{-s\min(\gamma,1/t')} (1+\abs{m})^{-n-\gamma}, \\
     &\min(\gamma,1/t')=\min(\gamma,1/2,1/p').
\end{split}
\end{equation*}

In the second factor, simply by H\"older's inequality, one may estimate the $L^t(S)$ norm by the $L^p(S)$ norm. Substituting back to \eqref{eq:shiftedT1toEst}, whose estimation was the original goal, it is found that
\begin{equation*}
\begin{split}
  &LHS\eqref{eq:shiftedT1toEst} \\
  &\lesssim (1+s)^{\delta_{\gamma,1/t'}\cdot 1/t'}2^{-s\min(\gamma,1/t')}(1+\abs{m})^{-n-\gamma} \\
  &\qquad\times\Big(\Exp_{\eps}\iint_{\R^n\times S}\Babs{\sum_{J\in\mathscr{D}}\eps_J
     \sum_{\ontop{I\subseteq J}{\ell(I)=2^{-s}\ell(J)}}\alpha_I h_I(y_J)1_J(x)}^p\ud x\ud\nu(y)\Big)^{1/p},
\end{split}
\end{equation*}
where $\nu$ is the product of the normalised Lebesgue measures on $S=\prod_{J\in\mathscr{D}}J$.

Let us reorganise the summation over $J\in\mathscr{D}$ as follows:
\begin{equation*}
  \sum_{J\in\mathscr{D}}=\sum_{j=0}^s\sum_{\substack{J\in\mathscr{D}\\ \log_2\ell(J)\equiv j (\operatorname{mod} s+1)}}
\end{equation*}
By standard estimates for random series (in the Banach space language, by the fact that $L^p$ has type $t$), we then have
\begin{equation*}
\begin{split}
  \Big(\Exp_{\eps}\iint &\Babs{\sum_{J\in\mathscr{D}}\eps_J
     \cdots}^p\ud x\ud\nu\Big)^{1/p} \\
  &\lesssim\Big\{\sum_{j=0}^s \Big(\Exp_{\eps}\iint\Babs{\sum_{\substack{J\in\mathscr{D}\\ \log_2\ell(J)\equiv j}}\eps_J
     \cdots}^p\ud x\ud\nu\Big)^{t/p}\Big\}^{1/t}.
\end{split}
\end{equation*}

Consider one of the new $J$-summations restricted by the condition that $\log_2\ell(J)\equiv j\mod s+1$. Since $h_I$ is constant on the dyadic cubes of side-length $\frac12\ell(I)$, one finds that each function
\begin{equation}\label{eq:fJtoApplyTangent}
   f_J:=\sum_{\ontop{I\subseteq J}{\ell(I)=2^{-s}\ell(J)}}\alpha_I h_I,
\end{equation}
obviously supported on $J$, is also constant on the cubes of side-length $2^{-s-1}\ell(J)$, and thus on all $K\in\mathscr{D}$ with $\ell(K)<\ell(J)$ and $\log_2\ell(K)\equiv\log_2\ell(J)\equiv j\mod s+1$.

These are exactly the conditions under which the following ``tangent martingale trick''  is applicable. Its essence goes back to McConnell~\cite{McConnell} in the context of decoupling estimates for stochastic integrals, and it was formulated as below in \cite[Theorem 6.1]{Hytonen:nonhomog} in order to facilitate its use in the estimation of singular integrals.

\begin{lemma}\label{lem:tangent}
Let $(E,\mathscr{M},\mu)$ be a $\sigma$-finite measure space equipped with partitions $\mathscr{A}_k\subset\mathscr{M}$ consisting of sets of finite positive measure, where $\mathscr{A}_{k+1}$ refines $\mathscr{A}_k$ for each $k\in\Z$. For each $A\in\mathscr{A}_k$, $k\in\Z$, let $f_A$ be a function supported on $A$ and constant on any $A'\in\mathscr{A}_{k+1}$, let $\mathscr{N}_A$ be the $\sigma$-algebra of $A$ for which all such functions are measurable, and let $\nu_A:=\mu(A)^{-1}\cdot\mu|_{\mathscr{N}_A}$.
 Let $(F,\mathscr{N},\nu)$ be the space $F:=\prod_{k\in\Z}\prod_{A\in\mathscr{A}_k}A$ equipped with the product $\sigma$-algebra and measure induced by the $\mathscr{N}_A$ and $\nu_A$. Denote a generic point of $F$ by $y=(y_A)_{A\in\mathscr{A}}$. Then the following norm equivalence holds with implied constants only depending on $p\in(1,\infty)$:
 \begin{align*}
  \int_E\Exp_{\eps} &\Babs{\sum_{k\in\Z}\eps_k\sum_{A\in\mathscr{A}_k}f_A(x)}^p\ud\mu(x) \\
  &\eqsim\iint_{E\times F}\Exp_{\eps}\Babs{\sum_{k\in\Z}\eps_k\sum_{A\in\mathscr{A}_k}f_A(y_A)1_A(x)}^p\ud\mu(x)\ud\nu(y).
\end{align*}
\end{lemma}

Indeed, on $(E,\mathscr{M},\ud\mu)=(\R^n,\mathscr{B}orel(\R^n),\ud x)$, take $\mathscr{A}_k:=\mathscr{D}_{k(s+1)+j}$ for a fixed $j\in\{0,1,\ldots,s\}$, and $f_A=f_J$ to be the functions defined in \eqref{eq:fJtoApplyTangent} for all $A=J\in\mathscr{A}_k$, $k\in\Z$. We apply Lemma~\ref{lem:tangent} separately for each $j$ to deduce that
\begin{align*}
  &\Big\{\sum_{j=0}^s\Big(\Exp_{\eps}\iint_{\R^n\times S}\Babs{\sum_{\substack{J\in\mathscr{D}\\ \log_2\ell(J)\equiv j}}\eps_J
     f_J(y_J)1_J(x)}^p\ud x\ud\nu(y)\Big)^{t/p}\Big\}^{1/t} \\
   &\lesssim\Big\{\sum_{j=0}^s\Big(\Exp_{\eps}\int_{\R^n}\Babs{\sum_{\substack{J\in\mathscr{D}\\ \log_2\ell(J)\equiv j}}\eps_J
     f_J(x)}^p\ud x\Big)^{t/p}\Big\}^{1/t}.
\end{align*}

With an application of H\"older's inequality and another standard estimate for the random series involving the exponent $q=\max(2,p)$ (in the Banach space language, the fact that $L^p$ has cotype $q$), this computation is continued with
\begin{equation*}
\begin{split}
  &\lesssim(s+1)^{1/t-1/q}\Big\{\sum_{j=0}^s\Big(\Exp_{\eps}\int_{\R^n}\Babs{\sum_{\substack{J\in\mathscr{D}\\ \log_2\ell(J)\equiv j}}\eps_J
     f_J(x)}^p\ud x\Big)^{q/p}\Big\}^{1/q} \\
  &\lesssim(s+1)^{1/t-1/q}\Big(\Exp_{\eps}\int_{\R^n}\Babs{\sum_{J\in\mathscr{D}}\eps_J 
     f_J(x)}^p\ud x\Big)^{1/p} \\
  &\lesssim(s+1)^{\abs{1/2-1/p}}\Big(\int_{\R^n}\Babs{\sum_{J\in\mathscr{D}}f_J(x)}^p\ud x\Big)^{1/p}
   \lesssim(s+1)^{\abs{1/2-1/p}}\Norm{f}{p},
\end{split}
\end{equation*}
where, in the last two steps, the signs $\eps_J$ were dropped by the unconditionality of the Haar functions and, recalling the definition of $f_J$ from \eqref{eq:fJtoApplyTangent}, it was observed that the resulting double sum over $J$ and $I$ is just a reorganisation of the summation over all $I\in\mathscr{D}$.

Substituting everything back, and observing that
\begin{equation*}
\begin{split}
  (1+s)^{\delta_{\gamma,1/t'}\cdot 1/t'}(1+s)^{\abs{1/2-1/p}}
  &\leq(1+s)^{1/t'+1/t-1/q} \\
  &=(1+s)^{1-1/q}\leq(1+s),
\end{split}
\end{equation*}
it is seen that \eqref{eq:shiftedT1mainEst}, and then Proposition~\ref{prop:shiftedT1}, has been completely proven. Indeed, a slightly smaller power for the factor $(1+s)$ would have been obtained, but this seems more like a curiosity, as this is only a fine-tuning of the decay rate of the exponential factor.

\section{The operator $\Psi_s$}

It remains to bound the operator $\Psi_s$, defined in \eqref{eq:defPsi} as
\begin{equation*}
   \Psi_s:=\sum_k(\id-\Exp_k)T_{4\cdot 2^{-k}}\D_{k+s}.
\end{equation*}
The relevant estimate to be proven is stated in the following. It is somewhat simpler than that for $\tilde{\Phi}_s$, in that the form of the upper bound does not depend on the exponent $p$, except via the implied multiplicative constant.

\begin{proposition}\label{prop:Psis}
Let $p\in(1,\infty)$ and $T$ be a Calder\'on--Zygmund operator with a normalised kernel. Then the operator $\Psi_s$ defined above satisfies
\begin{equation*}
  \Norm{\Psi_s}{p\to p}\lesssim(1+s)2^{-s\gamma}.
\end{equation*}
\end{proposition}

Since $\Exp_k^*=\Exp_k$ and $(\id-\Exp_k)(1)=0=\D_k(1)$, it follows that $\Psi_s(1)=\Psi_s^*(1)=0$. This suggests trying to deduce the norm bound for $\Psi_s$ from the special $T(1)$ theorem. However, the roughness of the conditional expectations implies that the kernel of $\Psi_s$ would not satisfy the standard estimates in their usual form. Instead, it will be checked the operator $\Psi_s$ satisfies certain intermediate estimates involved in Figiel's proof of the $T(1)$ theorem \cite{Figiel:90}, and this suffices by inspection of the mentioned proof.

Let me elaborate a little on this strategy. Figiel shows (under the assumption that $T(1)=T^*(1)=0$) that the Calder\'on--Zygmund standard estimates and the usual weak boundedness property for an operator $T$ imply the following estimates for its action on the Haar functions:
\begin{equation}\label{eq:CZHaar}
\begin{split}
  \abs{\pair{h_{I\dot+m}^{\theta}}{Th_I^{\zeta}}} &\lesssim(1+\abs{m})^{-n-\gamma},\\
  &I\in\mathscr{D},m\in\Z^n,(\theta,\zeta)\in\{0,1\}^{2n}\setminus(0,0).
\end{split}
\end{equation}
This in turn trivially implies that
\begin{equation}\label{eq:sumCZHaar}
\begin{split}
  \sum_{m\in\Z^n}\sup_{I\in\mathscr{D}}\abs{\pair{h_{I\dot+m}^{\theta}}{Th_I^{\zeta}}} &\log(2+\abs{m})\lesssim 1,\\
     &(\theta,\zeta)\in\{0,1\}^{2n}\setminus(0,0).
\end{split}
\end{equation}
Finally, Figiel proves the following step, which is most relevant for the present application:

\begin{lemma}[Figiel~\cite{Figiel:90}]\label{lem:FigielsT1}
Suppose that an operator $T$ satisfies $T(1)=T^*(1)=0$ and the estimate \eqref{eq:sumCZHaar}. Then $\Norm{T}{p\to p}\lesssim 1$ for all $p\in(1,\infty)$.
\end{lemma}

\begin{proof}
I sketch the argument from \cite{Figiel:90} for the convenience of the reader. For two functions $f$ and $g$ with a finite Haar expansion, we have
\begin{equation}\label{eq:FigielsSplit}
\begin{split}
  \pair{g}{Tf}
  &=\lim_{\substack{N\to+\infty\\ M\to-\infty}}[\pair{\Exp_N g}{T\Exp_N f}-\pair{\Exp_M g}{T\Exp_M f}] \\
  &=\sum_{k\in\Z}[\pair{\Exp_{k+1} g}{T\Exp_{k+1} f}-\pair{\Exp_k g}{T\Exp_k f}] \\ 
  &=\sum_{k\in\Z}[\pair{\D_k g}{T\D_k f}+\pair{\Exp_k g}{TD_k f}+\pair{\D_k g}{T\Exp_k f}].
\end{split}
\end{equation}

The first term is expanded as
\begin{align*}
  \sum_{k\in\Z} &\pair{\D_k g}{T\D_k f}
  =\sum_{k\in\Z^n}\sum_{\substack{I,J\in\mathscr{D} \\ \ell(I)=\ell(J)=k}}\sum_{\zeta,\eta}\pair{g}{h^{\zeta}_{J}}\pair{h^{\zeta}_J}{Th^{\eta}_I}\pair{h^{\eta}_I}{f} \\
  &=\sum_{m\in\Z^n}\sum_{\zeta}\Bpair{g}{\sum_{I\in\mathscr{D}}\sum_{\eta}{h^{\zeta}_{I\dot+m}}\pair{h^{\zeta}_{I\dot+m}}{Th^{\eta}_I}\pair{h^{\eta}_I}{f}} \\
  &=\sum_{m\in\Z^n}\sum_{\zeta}\Bpair{g}{\sum_{I\in\mathscr{D}}\sum_{\eta}T^{\zeta}_m\Theta^{\zeta}_m{h^{\eta}_I}\pair{h^{\eta}_I}{f}}
  =\sum_{m\in\Z^n}\sum_{\zeta}\Bpair{g}{T^{\zeta}_m\Theta^{\zeta}_m f},
\end{align*}
where $\Theta^{\zeta}_m$ and $T^{\zeta}_m$ are the linear operators acting on the Haar basis according to
\begin{equation*}
  \Theta^{\zeta}_m: h^{\eta}_I \to \pair{h^{\zeta}_{I\dot+m}}{Th^{\eta}_I} h^{\eta}_I,\qquad
  T^{\zeta}_m: h^{\eta}_I \to h^{\zeta}_{I\dot+m}.
\end{equation*}
They satisfy
\begin{equation*}
  \Norm{\Theta^{\zeta}_m}{p\to p}\lesssim \sup_{\substack{I\in\mathscr{D}\\ \eta\in\{0,1\}^n\setminus\{0\}}}\abs{\pair{h^{\zeta}_{I\dot+m}}{Th^{\eta}_I}},\qquad
  \Norm{T^{\zeta}_m}{p\to p}\lesssim\log(2+\abs{m}),
\end{equation*}
where the first estimate is essentially just the unconditionality of the Haar basis in $L^p(\R^n)$, and the second is \cite[Theorem~1]{Figiel:89}.

The second term on the right of \eqref{eq:FigielsSplit} can similarly be written as
\begin{align*}
  \sum_{k\in\Z} &\pair{\Exp_k g}{T\D_k f}
  =\sum_{m\in\Z^n}\Bpair{g}{\sum_{I\in\mathscr{D}}\sum_{\eta}{h^{0}_{I\dot+m}}\pair{h^{\zeta}_{I\dot+m}}{Th^{\eta}_I}\pair{h^{\eta}_I}{f}} \\
  &=\sum_{m\in\Z^n}\Bpair{g}{\sum_{I\in\mathscr{D}}\sum_{\eta}{[h^{0}_{I\dot+m}}-h^0_I]\pair{h^{\zeta}_{I\dot+m}}{Th^{\eta}_I}\pair{h^{\eta}_I}{f}}  \\
  &=\sum_{m\in\Z^n}\pair{g}{U_m\Theta^{0}_m f}
\end{align*}
with
\begin{equation*}
  \Theta^{0}_m: h^{\eta}_I \to \pair{h^{0}_{I\dot+m}}{Th^{\eta}_I} h^{\eta}_I,\qquad
  U_m: h^{\eta}_I \to h^{0}_{I\dot+m}-h^0_I,
\end{equation*}
since the total contribution of the subtracted correction terms is
\begin{equation*}
\begin{split}
  \sum_{I\in\mathscr{D}}\sum_{\eta} &\pair{g}{h^0_I}\Bpair{\sum_{m\in\Z^n}h^{0}_{I\dot+m}}{Th^{\eta}_I}\pair{h^{\eta}_I}{f} \\
  &=\sum_{I\in\mathscr{D}}\sum_{\eta}\pair{g}{h^0_I}\Bpair{\frac{1}{\abs{I}^{1/2}}}{Th^{\eta}_I}\pair{h^{\eta}_I}{f}=0
\end{split}
\end{equation*}
by the assumption that $T^*1=0$. The new operators again satisfy
\begin{equation*}
  \Norm{\Theta^0_m}{p\to p}\lesssim \sup_{\substack{I\in\mathscr{D}\\ \eta\in\{0,1\}^n\setminus\{0\}}}\abs{\pair{h^0_{I\dot+m}}{Th^{\eta}_I}},\qquad
  \Norm{U_m}{p\to p}\lesssim\log(2+\abs{m}),
\end{equation*}
by the unconditionality of the Haar basis and \cite[Theorem~1]{Figiel:89}.

The third term on the right of \eqref{eq:FigielsSplit} is essentially dual to the second; thus
\begin{equation*}
  \sum_{k\in\Z} \pair{\D_k g}{T\Exp_k f}
  =\sum_{m\in\Z^n}\pair{U_m\Theta_m g}{f}=\sum_{m\in\Z^n}\pair{g}{\Theta^*_m U_m^*f},
\end{equation*}
where
\begin{equation*}
  \Theta_m:h^{\eta}_I\to\pair{h^0_{I\dot+m}}{T^*h^{\eta}_I}h^{\eta}_I=\pair{h^{\eta}_I}{Th^0_{I\dot+m}}h^{\eta}_I
\end{equation*}
has norm
\begin{equation*}
  \Norm{\Theta_m^*}{p\to p}=\Norm{\Theta_m}{p'\to p'}
  \lesssim\sup_{\substack{I\in\mathscr{D}\\ \eta\in\{0,1\}^n\setminus\{0\}}}\abs{\pair{h^{\eta}_I}{Th^0_{I\dot+m}}}.
\end{equation*}

The assumption \eqref{eq:sumCZHaar} ensures that the formal expansion thus obtained,
\begin{equation*}
  T=\sum_{m\in\Z^n}\Big(\sum_{\zeta}T^{\zeta}_m\Theta^{\zeta}_m+U_m\Theta^0_m+\Theta^*_mU_m^*\Big)
\end{equation*}
converges in the $L^p(\R^n)$ operator norm, with the bound for $\Norm{T}{p\to p}$ given by the very quantity on the left of \eqref{eq:sumCZHaar} summed over the finitely many parameters $\theta,\zeta\in\{0,1\}^{2n}\setminus(0,0)$.
\end{proof}

Now we return to the problem at hand, i.e., proving Proposition~\ref{prop:Psis}. The operator $\Psi_s$ does not satisfy the standard estimates nor \eqref{eq:CZHaar} which, after all, is essentially just a dyadic version of the Calder\'on--Zygmund conditions. However, it will satisfy \eqref{eq:sumCZHaar}, with $(1+s)2^{-s\gamma}$ in place of the constant~$1$ on the right, which suffices to provide the same bound for $\Norm{\Psi_s}{p\to p}$ by Figiel's Lemma~\ref{lem:FigielsT1}. Besides giving what is needed here, this argument also shows the usefulness of \eqref{eq:sumCZHaar} as a weaker replacement of the Calder\'on--Zygmund standard estimates in the $T(1)$ theorem. I am not aware of any interesting earlier application of this condition.

Let us then turn to the realisation of the sketched programme, which requires the estimation of the Haar coefficients of $T$ appearing in \eqref{eq:sumCZHaar} with $\Psi_s$ in place of~$T$. The following computations will have the same spirit as those of Parcet \cite[Sec.~2.5]{Parcet:09}, but I feel that the present point of view of Haar coefficients somewhat simplifies matters.

Expanding the projections $\Exp_k$ and $\D_{k+s}$ in terms of the Haar functions, one gets
\begin{equation*}
\begin{split}
   \Psi_s &=\sum_k\Big(\id-\sum_{I\in\mathscr{D}_k}h_I^0\pair{h_I^0}{\cdot}\Big)
     T_{4\cdot 2^{-k}}\sum_{\ontop{J\in\mathscr{D}_{k+s}}{\eta\in\{0,1\}^n\setminus\{0\}}}
     h^{\eta}_J\pair{h^{\eta}_J}{\cdot} \\
  &=\sum_{J\in\mathscr{D}}\sum_{\eta}(T_{4\cdot 2^s\ell(J)}h^{\eta}_J)\pair{h^{\eta}_J}{\cdot} \\
   &\qquad-\sum_{\ontop{I,J\in\mathscr{D}}{\ell(I)=2^s\ell(J)}}\sum_{\eta}
     h_I^0\pair{h_I^0}{T_{4\cdot 2^s\ell(J)}h^{\eta}_J}\pair{h^{\eta}_J}{\cdot}.
\end{split}
\end{equation*}
Consequently, the orthogonality properties of the Haar functions imply, for $K,L\in\mathscr{D}$ with $\ell(K)=\ell(L)$ and $\theta,\zeta\in\{0,1\}^n\setminus\{0\}$, the following identities: (The three types of Haar coefficients of $\Psi_s$ listed are precisely those that one needs in~\eqref{eq:sumCZHaar}.)
\begin{equation}\label{eq:Haar11}
\begin{split}
  &\pair{h_K^{\theta}}{\Psi_s h_L^{\zeta}} \\
  &=\pair{h_K^{\theta}}{T_{4\cdot 2^s\ell(L)}h^{\zeta}_L} 
    -\sum_{\ontop{\ell(I)=2^s\ell(L)}{\phantom{\ell(I)}=2^s\ell(K)}}
     \pair{h_K^{\theta}}{h_I^0}\pair{h_I^0}{T_{4\cdot 2^s\ell(L)}h^{\zeta}_L} \\
  &=\pair{h_K^{\theta}}{T_{4\cdot 2^s\ell(L)}h^{\zeta}_L},
\end{split}
\end{equation}
and
\begin{equation}\label{eq:Haar01}
\begin{split}
  &\pair{h_K^{0}}{\Psi_s h_L^{\zeta}} \\
  &=\pair{h_K^{0}}{T_{4\cdot 2^s\ell(L)}h^{\zeta}_L}
   -\sum_{\ell(I)=2^s\ell(K)}
     \pair{h_K^{0}}{h_I^0}\pair{h_I^0}{T_{4\cdot 2^s\ell(L)}h^{\zeta}_L} \\
  &=\Bpair{h_K^{0}-2^{-ns/2}h_{K^{(s)}}^0}{T_{4\cdot 2^s\ell(L)}h^{\zeta}_L},
\end{split}
\end{equation}
and finally
\begin{equation}\label{eq:Haar10}
\begin{split}
  &\pair{h_K^{\theta}}{\Psi_s h_L^{0}} \\
  &=\sum_{J\supsetneq L,\eta}\pair{h_K^{\theta}}{T_{4\cdot 2^s\ell(J)}h^{\eta}_J}\pair{h^{\eta}_J}{h^0_L}\\
  &\qquad -\sum_{\ontop{I\subsetneq K,J\supsetneq L,\eta}{\ell(I)=2^s\ell(J)}}\pair{h_K^{\theta}}{h^0_I}
     \pair{h^0_I}{T_{4\cdot 2^s\ell(J)}h^{\eta}_J}\pair{h^{\eta}_J}{h^0_L} \\
  &=\sum_{J\supsetneq L,\eta}\pair{h_K^{\theta}}{T_{4\cdot 2^s\ell(J)}h^{\eta}_J}\pair{h^{\eta}_J}{h^0_L},\\
\end{split}
\end{equation}
where it was observed that the second summation is actually empty, since $I\subsetneq K$ and $J\supsetneq L$ imply $\ell(I)<\ell(K)=\ell(L)<\ell(J)$ which contradicts with $\ell(I)=2^s\ell(J)$.

\begin{lemma}\label{lem:truncTHaar}
For $L\in\mathscr{D}$ and $\zeta\in\{0,1\}^n\setminus\{0\}$, there holds
\begin{equation*}
\begin{split}
  \abs{T_{4\cdot 2^s\ell(L)} h_L^{\zeta}(x)}
  &\lesssim\frac{\ell(L)^{\gamma+n/2}}{\abs{x-y_L}^{n+\gamma}}1_{\abs{x-y_L}>3\cdot 2^s\ell(L)} \\
   &\qquad+\frac{\ell(L)^{n/2}}{\abs{x-y_L}^n}1_{\big|\abs{x-y_L}-4\cdot 2^s\ell(L)\big|<\ell(L)},
\end{split}
\end{equation*}
where $y_L$ is the centre of $L$.
\end{lemma}

\begin{proof}
By the cancellation of $h_L^{\zeta}$, one obtains
\begin{equation*}
\begin{split}
  &T_{4\cdot 2^s\ell(L)} h_L^{\zeta}(x) \\
  &=\int_{\abs{x-y}>4\cdot 2^s\ell(L)}K(x,y)h_L^{\zeta}(y)\ud y \\
  &=\int[K(x,y)1_{\abs{x-y}>4\cdot 2^s\ell(L)}-K(x,y_L)1_{\abs{x-y_L}>4\cdot 2^s\ell(L)}]h_L^{\zeta}(y)\ud y \\
  &=\int[K(x,y)-K(x,y_L)]1_{\abs{x-y}>4\cdot 2^s\ell(L)}h_L^{\zeta}(y)\ud y \\
  &\qquad+\int K(x,y_L)[1_{\abs{x-y}>4\cdot 2^s\ell(L)}-1_{\abs{x-y_L}>4\cdot 2^s\ell(L)}]h_L^{\zeta}(y)\ud y,
\end{split}
\end{equation*}
and hence, by the standard estimates and the size of the Haar functions,
\begin{equation*}
\begin{split}
  &\abs{T_{4\cdot 2^s\ell(L)} h_L^{\zeta}(x)} \\
  &\lesssim\int\frac{\ell(L)^{\gamma}}{\abs{x-y}^{n+\gamma}}1_{\abs{x-y}>4\cdot 2^s\ell(L)}\frac{1_L(y)}{\abs{L}^{1/2}}\ud y \\
  &\qquad+\int \frac{1}{\abs{x-y_L}^n}1_{\big|\abs{x-y_L}-4\cdot 2^s\ell(L)\big|<\ell(L)}\frac{1_L(y)}{\abs{L}^{1/2}}\ud y \\
  &\lesssim\frac{\ell(L)^{\gamma+n/2}}{\abs{x-y_L}^{n+\gamma}}1_{\abs{x-y_L}>3\cdot 2^s\ell(L)}
   +\frac{\ell(L)^{n/2}}{\abs{x-y_L}^n}1_{\big|\abs{x-y_L}-4\cdot 2^s\ell(L)\big|<\ell(L)},
\end{split}
\end{equation*}
which is the assertion.
\end{proof}

Now we estimate the quantity in \eqref{eq:Haar11}, and the first half of that in \eqref{eq:Haar01}.

\begin{lemma}
Let $K,L\in\mathscr{D}$ with $K=L\dot+m$ and $\theta\in\{0,1\}^n$, $\zeta\in\{0,1\}^n\setminus\{0\}$. Then
\begin{equation*}
  \abs{\pair{h_K^{\theta}}{T_{4\cdot 2^s\ell(L)} h_L^{\zeta}}}
  \lesssim\abs{m}^{-n-\gamma}1_{\abs{m}>2\cdot 2^s}
  +\abs{m}^{-n}1_{\big|\abs{m}-4\cdot 2^s\big|<2}.
\end{equation*}
\end{lemma}

\begin{proof}
Integrating the estimate of Lemma~\ref{lem:truncTHaar} against $\abs{h_K^{\zeta}}=\abs{K}^{-1/2}1_K=\abs{L}^{-1/2}1_K$, it follows that
\begin{equation*}
\begin{split}
  \abs{\pair{h_K^{\theta}}{T_{4\cdot 2^s\ell(L)} h_L^{\zeta}}}
  &\lesssim\frac{\ell(L)^{\gamma+n}}{\abs{x_K-y_L}^{n+\gamma}}1_{\abs{x_K-y_L}>2\cdot 2^s\ell(L)} \\
  &\qquad+\frac{\ell(L)^{n}}{\abs{x_K-y_L}^n}1_{\big|\abs{x_K-y_L}-4\cdot 2^s\ell(L)\big|<2\ell(L)},
\end{split}
\end{equation*}
where $x_K$ is the centre of $K$. Substituting $x_K=y_L+\ell(L)m$, the assertion follows.
\end{proof}

The estimate required in \eqref{eq:sumCZHaar}, for $\theta,\zeta\in\{0,1\}^n\setminus\{0\}$, now follows from
\begin{equation}\label{eq:Figiel11}
\begin{split}
  &\sum_{m\in\Z^n}\sup_{L\in\mathscr{D}}\abs{\pair{h^{\theta}_{L+m\ell(L)}}{\Psi_s h^{\zeta}_L}}\log(2+\abs{m}) \\
  &\lesssim\sum_{\abs{m}>2\cdot 2^s}\abs{m}^{-n-\gamma}\log(2+\abs{m}) \\
  &\qquad+\sum_{\big|\abs{m}-4\cdot 2^s\big|<2}\abs{m}^{-n}\log(2+\abs{m}) \\
  &\lesssim (1+s)2^{-s\gamma}+(1+s)2^{-s}\lesssim(1+s)2^{-s\gamma}.
\end{split}
\end{equation}
In bounding the second series, it was observed that all the summands are of the order $(1+s)2^{-sn}$, and their number is of the order $2^{s(n-1)}$, as they are essentially on the surface of a cube of side-length $8\cdot 2^s$. The obtained estimate exhibits desired exponential decay in $s$.

One still requires analogous estimates for the series where one of $\theta$ and $\zeta$ is allowed to be zero. To this end, we first look at the second half of the quantity in \eqref{eq:Haar01} involving the Haar function $h^0_{K^{(s)}}$:

\begin{lemma}\label{lem:pairingKs}
Let $K,L\in\mathscr{D}$ with $K=L\dot+m$ and $\zeta\in\{0,1\}^n\setminus\{0\}$. Then
\begin{equation*}
\begin{split}
  2^{-ns/2} &\abs{\pair{h_{K^{(s)}}^0}{T_{4\cdot 2^s\ell(L)}h^{\zeta}_L}} \\
  &\lesssim\abs{m}^{-n-\gamma}1_{\abs{m}>2\cdot 2^s}+2^{-s(n+1)}1_{\big|\abs{m}-4\cdot 2^s\big|<(1+2^s)}.
\end{split}
\end{equation*}
\end{lemma}

\begin{proof}
One has to integrate the estimate of Lemma~\ref{lem:truncTHaar} against
\begin{equation*}
  2^{-ns/2}\abs{h_{K^{(s)}}^{0}}=2^{-ns}\abs{L}^{-1/2}1_{K^{(s)}}. 
\end{equation*}
The first term of the mentioned estimate admits the upper bound
\begin{equation*}
  2^{-ns}\frac{\ell(L)^{\gamma}\ell(K^{(s)})^n}{\abs{x_K-y_L}^{n+\gamma}}1_{\abs{x_K-y_L}>2\cdot 2^s\ell(L)}
  =\frac{\ell(L)^{\gamma+n}}{\abs{x_K-y_L}^{n+\gamma}}1_{\abs{x_K-y_L}>2\cdot 2^s\ell(L)},
\end{equation*}
which gives the desired form upon substituting $x_K=y_L+\ell(L)m$.

In estimating the second term, observe that
\begin{equation}\label{eq:setSize}
\begin{split}
  &\Babs{\Big\{x:\big|\abs{x-y_L}-4\cdot 2^s\ell(L)\big|<\ell(L)\Big\}} \\
  &\qquad\qquad\lesssim\ell(L)\big(2^s\ell(L)\big)^{n-1}=2^{s(n-1)}\ell(L)^n,
\end{split}
\end{equation}
and on this set one has $\ell(L)^{n/2}\abs{x-y_L}^{-n}\lesssim 2^{-sn}\ell(L)^{-n/2}$. Hence the integration against $2^{-ns}\ell(L)^{-n/2}1_{K^{(s)}}$ gives at most $2^{-s(n+1)}$. On the other hand, for the integration to give a non-zero result at all, the set in \eqref{eq:setSize} and $K^{(s)}$ must intersect, which implies that
\begin{equation*}
  \big|\abs{x_K-y_L}-4\cdot 2^s\ell(L)\big|<\ell(L)+\ell(K^{(s)})=(1+2^s)\ell(L),
\end{equation*}
and a combination of these observations gives the claim.
\end{proof}

Now everything is prepared for the verification of \eqref{eq:sumCZHaar} in the case when $\theta=0$. Note that the first half of $\pair{h_K^0}{\Psi_s h_L^{\zeta}}$ on the right side of \eqref{eq:Haar01} is estimated in the same way as in \eqref{eq:Figiel11}, with the same result. Also the first term on the right of the upper bound in Lemma~\ref{lem:pairingKs} was already estimated there. Hence it follows that
\begin{equation}\label{eq:Figiel01}
\begin{split}
  \sum_{m\in\Z^n} &\sup_{L\in\mathscr{D}}\abs{\pair{h_{L+\ell(L)m}^0}{\Psi_s h_L^{\zeta}}}\log(2+\abs{m}) \\
  &\lesssim(1+s)2^{-s\gamma}+\sum_{\big|\abs{m}-4\cdot 2^s\big|<(1+2^s)}2^{-s(n+1)}(1+s) \\
  &\lesssim(1+s)[2^{-s\gamma}+2^{sn}2^{-s(n+1)}]\lesssim(1+s)2^{-s\gamma},
\end{split}
\end{equation}
where it simply used that the number of the summands is of the order $2^{sn}$.

It remains to estimate the third type of Haar coefficients of $\Psi_s$, namely those in \eqref{eq:Haar10}. (Note that for the usual Calder\'on--Zygmund operators, with assumptions symmetric with respect to the operator and its adjoint, one could have simply resorted to the symmetry and the case \eqref{eq:Haar01} which was already handled.)

\begin{lemma}\label{lem:Figiel10}
For $J,K,L\in\mathscr{D}$, where $L=K\dot+m$ and $J\supsetneq L$, and $\theta,\eta\in\{0,1\}^n\setminus\{0\}$, there holds
\begin{equation*}
\begin{split}
   \abs{\pair{h_K^{\theta}}{T_{4\cdot 2^s\ell(J)}h^{\eta}_J}\pair{h^{\eta}_J}{h^0_L}}
   &\lesssim\abs{m}^{-n-\gamma}1_{\abs{m}>2\cdot 2^{s+j}} \\
   &\qquad+2^{-sn}2^{-j(n+1)}1_{\big|\abs{m}-4\cdot 2^{s+j}\big|<1+2^j},
\end{split}
\end{equation*}
where $j=\log_2\big(\ell(J)/\ell(K)\big)$.
\end{lemma}

\begin{proof}
Let us start by observing that
\begin{equation*}
  \pair{h_K^{\theta}}{T_{4\cdot 2^s\ell(J)}h^{\eta}_J}
  =\pair{T_{4\cdot 2^{s+j}\ell(K)}^*h_K^{\theta}}{h^{\eta}_J},
\end{equation*}
where adjoint truncated singular integral $T_{4\cdot 2^{s+j}\ell(K)}^*$ satisfies exactly the same assumptions as $T_{4\cdot 2^{s+j}\ell(K)}$. Hence Lemma~\ref{lem:truncTHaar} shows that the first factor of the above pairing is pointwise dominated by
\begin{equation*}
  \frac{\ell(K)^{\gamma+n/2}}{\abs{x-x_K}^{n+\gamma}}1_{\abs{x-x_K}>3\cdot 2^{s+j}\ell(K)}
  +\frac{\ell(K)^{n/2}}{\abs{x-x_K}^n}1_{\big|\abs{x-x_K}-4\cdot 2^{s+j}\ell(K)\big|<\ell(K)}.
\end{equation*}

Integrating this bound against $\abs{h^{\eta}_J}=\ell(J)^{-n/2}1_J$, where $J\supsetneq L$, it follows that
\begin{equation*}
\begin{split}
  &\abs{\pair{T_{4\cdot 2^{s+j}\ell(K)}^*h_K^{\theta}}{h^{\eta}_J}} \\
  &\lesssim\frac{\ell(K)^{\gamma+n/2}\ell(J)^{n/2}}{\abs{y_L-x_K}^{n+\gamma}}1_{\abs{y_L-x_K}>2\cdot 2^s\ell(J)}\\
  &\qquad+\frac{\ell(K)^{n/2}}{(2^s\ell(J))^n}\frac{\ell(K)\ell(J)^{n-1}}{\ell(J)^{n/2}}
     1_{\big|\abs{y_L-x_K}-4\cdot 2^s\ell(J)\big|<\ell(K)+\ell(J)},
\end{split}
\end{equation*}
where a crucial observation was that the intersection of the cube $J$ (of length $\ell(J)$ in each coordinate direction) and the set 
\begin{equation*}
  \big\{x:\big|\abs{x-x_K}-4\cdot 2^{s+j}\ell(K)\big|<\ell(K)\big\}
\end{equation*}
(which has locally width $2\ell(K)$ in one of the coordinate directions) has measure at most of the order $\ell(K)\ell(J)^{n-1}$.

Multiplying the previous estimate by $\abs{\pair{h^{\eta}_J}{h^0_L}}=2^{-jn/2}$ and substituting $y_L=x_K+\ell(K)m$, the assertion follows.
\end{proof}

By using Lemma~\ref{lem:Figiel10} in order to estimate the expression in \eqref{eq:Haar10}, it follows that
\begin{equation}\label{eq:Figiel10}
\begin{split}
  &\sum_{m\in\Z^n}\sup_{K\in\mathscr{D}}\abs{\pair{h^{\theta}_K}{\Psi_s h^0_{K+m\ell(K)}}}\log(2+\abs{m}) \\
  &\lesssim\sum_{m\in\Z^n}\sum_{j=1}^{\infty}\Big(
     \abs{m}^{-n-\gamma}1_{\abs{m}>2\cdot 2^{s+j}} \\
  &\phantom{\lesssim\sum_{m\in\Z^n}\sum_{j=1}^{\infty}\Big(}
     +2^{-sn}2^{-j(n+1)}1_{\big|\abs{m}-4\cdot 2^{s+j}\big|<1+2^j}\Big)\log(2+\abs{m}) \\
  &=\sum_{j=1}^{\infty}\sum_{m\in\Z^n}\cdots \\
  &\lesssim\sum_{j=1}^{\infty}\Big(2^{-\gamma(s+j)}+2^{-sn}2^{-j(n+1)}\cdot 2^j(2^{s+j})^{n-1}\Big)(s+j) \\
  &=\sum_{j=1}^{\infty}\Big(2^{-\gamma(s+j)}+2^{-(s+j)}\Big)(s+j)\lesssim (1+s)2^{-\gamma s}.
\end{split}
\end{equation}
The estimates \eqref{eq:Figiel11}, \eqref{eq:Figiel01} and \eqref{eq:Figiel10} provide the required bound \eqref{eq:sumCZHaar}, with $\Psi_s$ in place of $T$ and $(1+s)2^{-s\gamma}$ in place of $1$. With Figiel's \cite{Figiel:90} proof of the $T(1)$ theorem, this implies the assertion of Proposition~\ref{prop:Psis}.

\section{A vector-valued extension}

An inspection of the proof of Theorem~\ref{thm:pseudoloc} provides the following vector-valued extension. It involves the notion of \emph{type} of a Banach space; recall that $X$ has type $t\in(1,2]$ if the randomised series enjoy the improved triangle inequality
\begin{equation*}
  \Exp_{\eps}\Bnorm{\sum_{j=1}^k\eps_j x_j}{X}\lesssim\Big(\sum_{j=1}^k\norm{x_j}{X}^t\Big)^{1/t}.
\end{equation*}
(In this section, the implicit constants involved in the notation ``$\lesssim$'' are also allowed to depend on the Banach space $X$ and its type $t$, in addition to $n$, $p$, and $\gamma$.) As the scalar field has type $2$, the following statement is indeed recognised, up to the polynomial factor, as a generalisation of Theorem~\ref{thm:pseudoloc}.

\begin{corollary}\label{cor:pseudoloc}
Let $X$ be a UMD space of type $t\in(1,2]$, let $p\in(1,\infty)$, $f\in L^p(\R^n;X)$, and $s\in\N$. Then every Calder\'on--Zygmund operator $T$ with a normalised kernel satisfies
\begin{equation}\label{eq:pseudolocX}
  \Big(\int_{\Sigma_{f,s}^c}\abs{Tf(x)}_X^p\ud x\Big)^{1/p}
  \lesssim (1+s)^2 2^{-s\min(\gamma,1/t',1/p')}\Norm{f}{p}.
\end{equation}
If, moreover, $T$ is bounded and normalised, this estimate also holds with $\Sigma_{f,s}$ replaced by
\begin{equation}\label{eq:QfsX}
  100\cdot 2^{s[1+\min(\gamma,1/t',1/p')\cdot p'/n]}Q_{f,s},
\end{equation}
where $Q_{f,s}$ is any cube such that
\begin{equation*}
  \Norm{1_{Q_{f,s}^c}f}{p}\leq (1+s)^2 2^{-s\min(\gamma,1/t',1/p')}\Norm{f}{p}.
\end{equation*}
\end{corollary}

Indeed, most parts of the proof of Theorem~\ref{thm:pseudoloc} employed methods and results which were developed for the UMD space -valued situation from the beginning, so that they can be simply repeated in the present context. This is in particular the case for Figiel's $T(1)$ theorem \cite{Figiel:89,Figiel:90}, and also for the tangent martingale inequality \cite{Hytonen:nonhomog,McConnell}. However, a step which requires additional explanation is the estimate \eqref{eq:pullOutNorms}, based on Lemma~\ref{lem:pullOutNorms}.

The following distributional variant of Lemma~\ref{lem:averEst} will be needed.

\begin{lemma}\label{lem:averEstLorentz}
For $I,J\in\mathscr{D}$ with $\ell(I)\leq\ell(J)$ and $m\in\Z^n\setminus\{0\}$, there holds
\begin{equation*}
\begin{split}
  \frac{1}{\abs{J}}&\Babs{\Big\{x\in J:\Babs{\abs{J}^{1/2}
     \sum_{\ontop{K\subseteq J}{\ell(K)=\ell(I)}}\pair{h^0_{J\dot+m}}{Th^{\eta}_K}h^{\eta}_K(x)}>\lambda\Big\}} \\
  &\lesssim\begin{cases}
      1_{[0,C(\ell(I)/\ell(J))^{\gamma}(1+\abs{m})^{-n-\gamma}]}(\lambda), & m\notin\{-1,0,1\}^n, \\
      \min\{\lambda^{-1/\gamma}\ell(I)/\ell(J),1\}\times 1_{[0,C]}(\lambda), & m\in\{-1,0,1\}^n\setminus\{0\}.
   \end{cases}
\end{split}
\end{equation*}
\end{lemma}

\begin{proof}
For $m\notin\{-1,0,1\}^n$, the above bound is just a reformulation of the $L^{\infty}$ estimate pointed out in the beginning of the proof of Lemma~\ref{lem:averEst}. Let then $m\in\{-1,0,1\}^n\setminus\{0\}$, so that $J\dot+m$ and $J$ are adjacent. For each $k=0,\ldots,\ell(J)/\ell(I)$, there are $O((\ell(J)/\ell(I))^{n-1})$ cubes $K$ with
\begin{equation*}
  \dist(K,\partial(J\dot+m))=k\ell(I),
\end{equation*}
and on such a $K$, Lemma~\ref{lem:HaarEst} gives
\begin{equation*}
\begin{split}
  &\Babs{\abs{J}^{1/2}\pair{h^0_{J\dot+m}}{Th^{\eta}_K}h^{\eta}_K(x)} \\
  &\qquad\lesssim\abs{J}^{1/2}\Big(\frac{\ell(I)}{\ell(J)}\Big)^{n/2}\Big(\frac{\ell(I)}{(1+k)\ell(I)}\Big)^{\gamma}\frac{1}{\abs{I}^{1/2}}
  =(1+k)^{-\gamma}.
\end{split}
\end{equation*}
Thus the number of cubes, where the value of the function exceeds $Ck^{-\gamma}$, is at most $C\min\{k,\ell(J)/\ell(I)\}\big(\ell(J)/\ell(I)\big)^{n-1}$ for $k\in\Z_+$, and hence their normalised measure is at most $C\min\{k\ell(I)/\ell(J),1\}$. Also notice that the value of the function is never bigger than some absolute constant $C$. The change of variable into $\lambda:=Ck^{-\gamma}$, thus $k=(C/\lambda)^{1/\gamma}$, proves the assertion.
\end{proof}

Lemma~\ref{lem:pullOutNorms} has the following analogue, based on a result of Veraar and the author \cite[Lemma 3.1]{HV:smooth}.

\begin{lemma}\label{lem:pullOutNormsX}
Let $X$ have type $t\in(1,2]$. Then for $f_j\in L^t(\mu;X)$ and $\phi_j\in L^{t',1}(\mu)$ (the Lorentz space), there holds
\begin{equation*}
\begin{split}
  \Exp_{\eps}&\Bnorm{\sum_j\eps_j\int_S\phi_j(s)f_j(s)\ud\mu(s)}{X} \\
  &\lesssim\int_0^{\infty}\sup_j\mu(\{s:\abs{\phi_j(s)}>\lambda\})^{1/t'}\ud\lambda\cdot
   \Exp_{\eps}\BNorm{\sum_j\eps_j f_j}{t}.
\end{split}
\end{equation*}
\end{lemma}

Note that, without the supremum over $j$, the integral would be the Lorentz $L^{t',1}(\mu)$ norm of $\phi_j$.

\begin{proof}
Using the duality of the randomised norms, the left side is comparable to
\begin{equation*}
  \sup\Big\{\Babs{\sum_j\Bpair{x_j^*}{\int_S\phi_j(s)f_j(s)\ud\mu(s)}}:
     \Exp_{\eps}\Bnorm{\sum_j\eps_j x_j^*}{X^*}\leq 1\Big\}.
\end{equation*}
Then
\begin{equation*}
\begin{split}
  &\Babs{\sum_j\Bpair{x_j^*}{\int_S\phi_j(s)f_j(s)\ud\mu(s)}} \\
  &=\Babs{\int_S\Exp_{\eps}\Bpair{\sum_j\eps_j\phi_j(s)x_j^*}{\sum_j\eps_j f_j(s)}\ud\mu(s)} \\
  &\leq\BNorm{\sum_j\eps_j\phi_j(\cdot)x_j^*}{L^{t'}(\mu\otimes\prob;X^*)}
   \BNorm{\sum_j\eps_j f_j}{L^t(\mu\otimes\prob;X)}
\end{split}
\end{equation*}
where the second factor is comparable to the second factor in the assertion. Since the dual space $X^*$ has cotype $t'\in[2,\infty)$, the first factor can be estimated by \cite[Lemma 3.1]{HV:smooth}, which gives
\begin{equation*}
\begin{split}
  &\BNorm{\sum_j\eps_j\phi_j(\cdot)x_j^*}{L^{t'}(\mu\otimes\prob;X^*)} \\
  &\lesssim\int_0^{\infty}\sup_j\mu(\{s:\abs{\phi_j(s)}>\lambda\})^{1/t'}\ud\lambda\cdot
    \Exp_{\eps}\Bnorm{\sum_j\eps_j x_j^*}{X^*}.
\end{split}
\end{equation*}
This completes the proof.
\end{proof}

In \eqref{eq:pullOutNorms}, the right side is now replaced by
\begin{equation}\label{eq:pullOutNormsX}
\begin{split}
  \int_0^{\infty}\sup_{J\in\mathscr{D}} &\Big(\frac{1}{\abs{J}}
    \Babs{\Big\{x\in J:\Babs{\abs{J}^{1/2}\sum_{\ontop{I\subseteq J}{\ell(I)=2^{-s}\ell(J)}}\lambda_I h_I(x)}>\lambda\Big\}}
     \Big)^{1/u'}\ud\lambda \\
  &\times\Exp_{\eps}\BNorm{\sum_{J\in\mathscr{D}}\eps_J
    \sum_{\ontop{I\subseteq J}{\ell(I)=2^{-s}\ell(J)}}\alpha_I h_I(y_I)1_I(x)}{L^u(S)},
\end{split}
\end{equation}
where we choose $u:=\min(t,p)$. Recalling that $\lambda_I=\pair{h^0_{I^{(s)}\dot+m}}{Th^{\eta}_I}$, one finds from Lemma~\ref{lem:averEstLorentz} that
\begin{equation*}
\begin{split}
  \frac{1}{\abs{J}}
    &\Babs{\Big\{x\in J:\Babs{\abs{J}^{1/2}\sum_{\ontop{I\subseteq J}{\ell(I)=2^{-s}\ell(J)}}\lambda_I h_I(x)}>\lambda\Big\}} \\
  &\lesssim\begin{cases}
     1_{[0,C 2^{-s\gamma}(1+\abs{m})^{-n-\gamma}]}(\lambda), & m\notin\{-1,0,1\}^n, \\
     \min\{\lambda^{-1/\gamma}2^{-s},1\}\cdot 1_{[0,C]}(\lambda), & m\in\{-1,0,1\}^n\setminus\{0\}.
   \end{cases}
\end{split}
\end{equation*}
where the right side is independent of $J\in\mathscr{D}$. Hence the first factor in \eqref{eq:pullOutNormsX} is dominated by
\begin{equation*}
  \int_0^{2^{-s\gamma}}\ud\lambda+\int_{2^{-s\gamma}}^{C}(\lambda^{-1/\gamma}2^{-s})^{1/u'}\ud\lambda
  \lesssim 2^{-s\min(\gamma,1/u')}(1+s)^{\delta_{\gamma,1/u'}}
\end{equation*}
if $m\in\{-1,0,1\}^n\setminus\{0\}$, and by $C2^{-s\gamma}(1+\abs{m})^{-n-\gamma}$ for $m\notin\{-1,0,1\}^n$. Substituting back to \eqref{eq:pullOutNormsX} and recalling that this was the replacement of the right side of \eqref{eq:pullOutNorms} in the vector-valued situation under consideration, it follows that
\begin{equation*}
\begin{split}
  LHS\eqref{eq:shiftedT1pointwise}
  \lesssim (1+ &s)^{\delta_{\gamma,1/u'}}2^{-s\min(\gamma,1/u')}(1+\abs{m})^{-n-\gamma}\times \\
   &\times\Exp_{\eps}\BNorm{\sum_{J\in\mathscr{D}}\eps_J
    \sum_{\ontop{I\subseteq J}{\ell(I)=2^{-s}\ell(J)}}\alpha_I h_I(y_I)1_I(x)}{L^u(S)}.
\end{split}
\end{equation*}
The proof of Corollary~\ref{cor:pseudoloc} is then completed just like that of Theorem~\ref{thm:pseudoloc}; now $L^p(\R^n;X)$ has type $\min(t,p)$ and some cotype $q\in[2,\infty)$, and one checks that this suffices to get the bound with the asserted quadratic polynomial factor instead of the linear one in Theorem~\ref{thm:pseudoloc}.


\begin{thebibliography}{1}

\bibitem{Figiel:89}
Tadeusz Figiel.
\newblock On equivalence of some bases to the {H}aar system in spaces of
  vector-valued functions.
\newblock {\em Bull. Polish Acad. Sci. Math.}, 36(3-4):119--131 (1989), 1988.

\bibitem{Figiel:90}
Tadeusz Figiel.
\newblock Singular integral operators: a martingale approach.
\newblock In {\em Geometry of {B}anach spaces ({S}trobl, 1989)}, volume 158 of
  {\em London Math. Soc. Lecture Note Ser.}, pages 95--110. Cambridge Univ.
  Press, Cambridge, 1990.

\bibitem{Hytonen:nonhomog}
Tuomas Hyt\"onen.
\newblock The vector-valued non-homogeneous {$Tb$} theorem.
\newblock Preprint, arXiv:0809.3097, 2008.

\bibitem{HV:smooth}
Tuomas Hyt\"onen and Mark Veraar.
\newblock {$R$}-boundedness of smooth operator-valued functions.
\newblock {\em Integral Equations Operator Theory}, 63(3):373--402, 2009.

\bibitem{McConnell}
Terry~R. McConnell.
\newblock Decoupling and stochastic integration in {UMD} {B}anach spaces.
\newblock {\em Probab. Math. Statist.}, 10(2):283--295, 1989.

\bibitem{MeiPar}
Tao Mei and Javier Parcet.
\newblock Pseudo-localization of singular integrals and noncommutative
  {L}ittlewood-{P}aley inequalities.
\newblock {\em Int. Math. Res. Not. IMRN}, (8):1433--1487, 2009.

\bibitem{Parcet:09}
Javier Parcet.
\newblock Pseudo-localization of singular integrals and noncommutative
  {C}alder\'on-{Z}ygmund theory.
\newblock {\em J. Funct. Anal.}, 256(2):509--593, 2009.

\end{thebibliography}
\end{document}